\documentclass[final]{siamltex}

\newtheorem{remark}{Remark}[section]
\usepackage[cmex10]{amsmath}
\usepackage{graphics}           
\usepackage{graphicx}
\usepackage{subfigure}
\usepackage{amssymb}
\title{Set-Valued return function and generalized solutions for Multiobjective Optimal Control  Problems (MOC) \footnotemark[1]}

\author{A. Guigue \footnotemark[1]}

\begin{document}

\maketitle

\renewcommand{\thefootnote}{\fnsymbol{footnote}}
\footnotetext[1]{Department of Mathematics, The University of British Columbia, Room 121, 1984 Mathematics Road,
Vancouver, B.C., Canada, V6T 1Z2 (aguigue@math.ubc.ca).}

\begin{abstract}
In this paper, we consider a multiobjective optimal control problem where the
preference relation in the objective space is defined in terms of a pointed convex
cone containing the origin, which defines generalized Pareto optimality. For
this problem, we introduce the set-valued return function $V$ and provide
a unique characterization for $V$ in terms of contingent derivative and coderivative for
set-valued maps, which extends  two previously introduced notions of generalized solution to the Hamilton-Jacobi
equation for single objective optimal control problems.
\end{abstract}



\pagestyle{myheadings} \thispagestyle{plain} \markboth{A. GUIGUE}{GENERALIZED SOLUTIONS FOR (MOC)}

\section{Introduction}

Many engineering applications, such as trajectory planning for spacecraft~\cite{MOC_appl_6} and robotic manipulators~\cite{Guigue10}, continuous casting of steel~\cite{MOC_appl_5}, etc., can lead to an optimal control formulation where $p$ objective functions ($p > 1$) need to be optimized simultaneously. For multiobjective
optimization problems, optimality is determined by the preference of the decision maker, which is expressed
in terms of a binary relation. In this paper, we will consider the preference defined in terms of a pointed convex cone $P \subset \mathbf{R}^p$ containing the origin \cite{Yu74_moo}, which, when $P=\mathbf{R}^p_+$, yields the well-known Pareto optimality. The derivation of optimality conditions for multiobjective optimal control  problems (MOC)
with more general preferences than the one considered in this paper has been a subject of recent interest. Mostly,
the approach \cite{Bellaassali04_opt, Kien08_opt, Zhu01_opt} has been to derive necessary optimality conditions. In this paper, we take a different approach, treating (MOC) in the framework of dynamic programming~\cite{OptConTh8_book}. Namely, for an autonomous multiobjective optimal control problem  with fixed end time  where the dynamics are governed by a differential equation, we first define a set-valued return function $V$. Using the concepts of contingent derivative \cite[p.~181]{SVA_book} and coderivative \cite{Mordukhovich93_moo} for set-valued maps, we then provide a unique characterization for $V$, which extends the notions of generalized solution to the Hamilton-Jacobi equation for single objective optimal control problems introduced in  \cite[p.~454]{OptConTh5_book}.\\

A work related  to ours is \cite{Caroff08_MOC}, where the same multiobjective optimal control problem is considered, but the preference is  given  by the lexicographical order. In \cite{Caroff08_MOC}, a ``vector Value function'' for the problem is first defined. Two notions of solution to a suitable system of Hamilton-Jacobi equations are then introduced: the notion of contingent solution, which uses the concept of contingent vector epiderivative,  and the notion of vector viscosity solution. The main result in \cite{Caroff08_MOC} is to show that the ``vector Value function'' is the unique vector extended lower semicontinuous contingent and vector viscosity solution. This is also the program
followed in  this paper for (MOC). More precisely, we start by defining the set-valued return function $V(\cdot,\cdot): [0,T] \times \mathbf{R}^n  \rightarrow 2^{\mathbf{R}^p}$ as the set-valued map associating with each time $t \in [0,T]$ and state $\mathbf{x} \in \mathbf{R}^n$ the set of generalized Pareto optimal elements (which we call minimal elements) in the objective space $Y(t,\mathbf{x})$, where $Y(t,\mathbf{x})$ is the set of all possible values that can be taken by the vector-valued objective function for trajectories starting at~$\mathbf{x}$ at time $t$. The set-valued return function $V$ is an \emph{extremal element map}, i.e., a map whose set values only contain minimal elements. For such maps, using the concepts of contingent derivative and coderivative for set-valued maps, we then extend the two notions of generalized solution to the Hamilton-Jacobi equation for single objective optimal control problems  introduced in \cite[p.~454]{OptConTh5_book}. We call these two extended notions generalized contingent solution and generalized proximal solution for (MOC). Our main result is to show that $V$ is the unique set-valued map $W$
 generalized contingent and proximal solution such that the set-valued map $W_\uparrow := W+P$ is outer semicontinuous. The proof of this result is made of several steps. First, we show that $V$ is a generalized contingent solution and that $V_\uparrow$ is outer semicontinuous. In the process of this proof, we state the multiobjective dynamic programming equation for~(MOC). Next, we show  that generalized contingent solution implies generalized proximal solution. We finally obtain uniqueness for extremal element maps using invariance theorems \cite{OptConTh5_book}.\\

This paper is organized as follows. In \S \ref{s:P}, we detail the class of multiobjective optimal control problems considered and define the set-valued return function $V$. In \S \ref{s:mathprel}, we then provide useful mathematical definitions. In \S \ref{s:moogeneral}, we discuss the concept of optimality  in multiobjective optimization and present several properties of the minimal element set. Also, we provide a general sufficient condition for  $V$ to be a Lipschitz set-valued map. In \S \ref{s:moo}, we state the multiobjective dynamic programming equation for (MOC). In \S \ref{s:hjb}, we  present the notion of generalized contingent solution.  In \S \ref{s:setvalued}, we prove that $V$ is a generalized contingent solution and that $V_\uparrow$ is outer semicontinuous. Finally, in \S \ref{s:proximal},
we  present the notion of generalized proximal  solution, show that generalized contingent solution implies generalized proximal solution, and prove that $V$ is the unique generalized proximal solution.

\section{A multiobjective finite-horizon optimal control problem (MOC) \cite{OptConTh8_book}} \label{s:P}

Consider the evolution over a fixed finite time interval $I =
[0,T] \ ( 0 < T < \infty)$ of an autonomous dynamical system whose
$n$-dimensional state dynamics are given by a continuous function
$\mathbf{f}(\cdot,\cdot):  \mathbf{R}^n \times U
\rightarrow \mathbf{R}^{n}$, where the control space $U$ is a
nonempty compact subset of $\mathbf{R}^m$.
The function $\mathbf{f}(\cdot,\mathbf{u})$ is assumed to be
Lipschitz, i.e., some $K_{\mathbf{f}} > 0$ obeys
\begin{equation} \label{eq:lipf}
\forall \mathbf{u} \in U, \ \forall
\mathbf{x_1},\mathbf{x_2} \in \mathbf{R}^{n}, \ \|
\mathbf{f}(\mathbf{x_1},\mathbf{u}) -
\mathbf{f}(\mathbf{x_2},\mathbf{u}) \| \leq K_{\mathbf{f}} \|
\mathbf{x_1} - \mathbf{x_2} \|.
\end{equation}
We also assume that the function $\mathbf{f}$ is uniformly bounded, i.e., some
$M_\mathbf{f} > 0$ obeys
\begin{equation} \label{eq:fbounded}
\forall \mathbf{x} \in \mathbf{R}^n, \ \forall \mathbf{u} \in U, \
\| \mathbf{f}(\mathbf{x},\mathbf{u}) \| \leq M_\mathbf{f}.
\end{equation}
A control $\mathbf{u}(\cdot): I \rightarrow U$ is a bounded,
Lebesgue measurable function. The set of  controls
is denoted by $\mathcal{U}$. The continuity of $\mathbf{f}$ and the Lipschitz condition (\ref{eq:lipf})
guarantee that, given any $t \in I$, initial state $\mathbf{x} \in \mathbf{R}^n$, and
control $\mathbf{u}(\cdot) \in \mathcal{U}$,  the system
of differential equations governing the dynamical system,
\begin{equation} \label{eq:nec1}
\left \{ \begin{array}{lcl} \mathbf{\dot{x}}(s)& = & \mathbf{f}(\mathbf{x}(s),\mathbf{u}(s)), \ t \leq s  \leq T,\\
\mathbf{{x}}(t) & = & \mathbf{x},\end{array}\right.
\end{equation}
has a unique solution, called a trajectory and denoted $s \rightarrow \mathbf{x}(s;t,\mathbf{x},\mathbf{u}(\cdot))$. Using the Gronwall inequality, the following estimate between trajectories can be obtained.\\
\begin{proposition}[{\rm Estimate between trajectories}]
\label{prop:trajestim} Let $t \in I$, $\mathbf{x_1},\mathbf{x_2} \in \mathbf{R}^n$, and $\mathbf{u}(\cdot) \in \mathcal{U}$. Then,
\begin{displaymath}
\forall s \in [t,T], \ \|\mathbf{x}(s;t,\mathbf{x_1},\mathbf{u}(\cdot))-\mathbf{x}(s;t,\mathbf{x_2},\mathbf{u}(\cdot))\| \leq \exp(K_{\mathbf{f}} (s-t))\|\mathbf{x_1}-\mathbf{x_2}\|.
\end{displaymath}
\end{proposition}

The cost of a trajectory over $[t,T], \ t \in I,$ is given by a
$p$-dimensional vector function $\mathbf{J}(\cdot,\cdot,\cdot,\cdot): I
\times I \times \mathbf{R}^n \times \mathcal{U} \rightarrow
\mathbf{R}^p$,
\begin{equation} \label{eq:obj}
\mathbf{J}(t,T,\mathbf{x},\mathbf{u}(\cdot)) = \int_{t}^{T}
\mathbf{L}(\mathbf{x}(s;\mathbf{x},\mathbf{u}(\cdot)),\mathbf{u}(s)) \ \mathrm{d}s,
\end{equation}
where the $p$-dimensional vector function
$\mathbf{L}(\cdot,\cdot): \mathbf{R}^n \times U
\rightarrow \mathbf{R}^p$ is assumed to be continuous. For simplicity,  no terminal cost  is included in (\ref{eq:obj}). We assume that the function $\mathbf{L}$ is uniformly bounded, i.e., some
$M_\mathbf{L} \geq 0$ obeys
\begin{equation} \label{eq:Lbounded}
\forall \mathbf{x} \in \mathbf{R}^n, \ \forall \mathbf{u} \in U, \
\| \mathbf{L}(\mathbf{x},\mathbf{u}) \| \leq M_\mathbf{L},
\end{equation}
and that the function $\mathbf{L}(\cdot,\mathbf{u})$ satisfies a Lipschitz condition, i.e., some $K_\mathbf{L} \geq 0$ obeys
\begin{equation} \label{eq:LLipschitz}
\forall \mathbf{u} \in U, \ \forall
\mathbf{x_1},\mathbf{x_2} \in \mathbf{R}^n, \ \|
\mathbf{L}(\mathbf{x_1},\mathbf{u}) -
\mathbf{L}(\mathbf{x_2},\mathbf{u}) \| \leq K_\mathbf{L} \| \mathbf{x_1}
- \mathbf{x_2} \|.
\end{equation}
Using  (\ref{eq:LLipschitz}) and Proposition \ref{prop:trajestim}, the following estimate between
the cost of two trajectories can be obtained.\\
\begin{proposition}[{\rm Estimate between costs}]
\label{prop:costestim}
Let $t_1,t_2 \in I$,  $\mathbf{x_1},\mathbf{x_2} \in \mathbf{R}^n$, and $\mathbf{u}(\cdot) \in \mathcal{U}$. Then,
\begin{displaymath}
\|\mathbf{J}(t_1,T,\mathbf{x_1},\mathbf{u}(\cdot)) - \mathbf{J}(t_2,T,\mathbf{x_2},\mathbf{u}(\cdot)) \| \leq
\frac{K_{\mathbf{L}}}{K_{\mathbf{f}}}  \exp(K_{\mathbf{f}}T) \|\mathbf{x_1}-\mathbf{x_2}\|   + M_{\mathbf{L}} |t_1-t_2|.
\end{displaymath}
\end{proposition}

The objective space $Y(t,\mathbf{x})$  for (MOC) is defined as the set of
all possible costs~(\ref{eq:obj}):
\begin{displaymath}
\ Y(t,\mathbf{x}) = \bigg\{
\mathbf{J}(t,T,\mathbf{x},\mathbf{u}(\cdot)), {\mathbf{u}(\cdot) \in
\mathcal{U}} \bigg\}.
\end{displaymath}
From  $(\ref{eq:Lbounded})$, it follows that the set $Y(t,\mathbf{x})$ is bounded
(by $M_\mathbf{L}T$), however $Y(t,\mathbf{x})$ is not necessarily closed.
Using Proposition \ref{prop:costestim},  the following estimate between
objective spaces can be obtained.\\
\begin{corollary} [{\rm Estimate between objective spaces}] \label{prop:objspesti}
Let $t_1,t_2 \in I$ and  $\mathbf{x_1},\mathbf{x_2} \in \mathbf{R}^n$. Then,
$$Y(t_1,\mathbf{x_1}) \subset Y(t_2,\mathbf{x_2}) + \bigg(\frac{K_{\mathbf{L}}}{K_{\mathbf{f}}} \exp(K_{\mathbf{f}}T) \|\mathbf{x_1}-\mathbf{x_2}\|  +
M_{\mathbf{L}} |t_1-t_2|\bigg) \mathbf{B},$$
where $\mathbf{B}(\mathbf{x},l)$ is the closed ball
centered at $\mathbf{x}$ with radius $l$ and $\mathbf{B} = \mathbf{B}(\mathbf{0},1)$.
\end{corollary}\\ \\
In \S \ref{ss:lip}, we will need the norm $\|\cdot\|$ in $\mathbf{R}^p$ to be Euclidian. Let $\langle \cdot,\cdot \rangle$ be the associated inner product.\\

The set-valued return function $V(\cdot,\cdot): I \times \mathbf{R}^n  \rightarrow
2^{\mathbf{R}^p}$ for (MOC) is defined as the set-valued map which associates with each time $t \in I$ and initial
state $\mathbf{x} \in \mathbf{R}^n$ the set of minimal elements of the objective space $Y(t,\mathbf{x})$
with respect to the ordering cone $P$, where the definition of a minimal element is postponed to \S \ref{ss:opti}:
\begin{equation} \label{eq:setvalued}
V(t,\mathbf{x}) = \mathcal{E}(\mathrm{cl}(Y(t,\mathbf{x})),P).
\end{equation}
The closure in (\ref{eq:setvalued}) is used to guarantee the existence of minimal elements (Proposition 3.5,
\cite{Guigue09}). Hence, $\forall t \in I, \ \forall \mathbf{x} \in \mathbf{R}^n, \ V(t,\mathbf{x}) \neq \emptyset$. \\

\begin{remark} \label{rem:valuefunction1dof} When $p=1$ and $P = \mathbf{R}_+$, \emph{(\ref{eq:setvalued})}
takes the form
$$V(t,\mathbf{x})  =  \bigg\{ \inf_{\mathbf{u}(\cdot) \in \mathcal{U}}  \int_{t}^{T}
L(\mathbf{x}(s;t,\mathbf{x},\mathbf{u}(\cdot)),\mathbf{u}(s)) \ \mathrm{d}s   \bigg\}.$$
Hence, $V(t,\mathbf{x}) = \{v(t,\mathbf{x})\},$ where $v(\cdot,\cdot)$ is the value function
for single objective optimal control problems \emph{\cite{OptConTh8_book,OptConTh5_book}}.
\end{remark}

\section{Mathematical preliminaries} \label{s:mathprel}

In this section, we recall some useful definitions and discuss general set-valued maps with values
in $\mathbf{R}^p$. Let $X$ be a normed linear space, $F$ be a set-valued map from $X$ to
$\mathbf{R}^p$, $\mathcal{K} = \{S \subset \mathbf{R}^p, \ S \neq \emptyset, \ S \ \mbox{compact} \}$ and
$\mathcal{M} = \{S \subset \mathbf{R}^p, \ S \neq \emptyset, \ S \ \mbox{bounded} \}$.\\

First, we provide the definition of the Hausdorff distance  between $M_1$ and $M_2$, where $M_1,M_2 \in \mathcal{M}$ \cite[p. 365]{SVA_book}.\\

\begin{definition}
 Given $M_1,M_2 \in \mathcal{M}$, the \emph{Hausdorff distance}  between $M_1$ and $M_2$ is
\begin{displaymath}
\mathcal{H}(M_1,M_2)= \max \{ \sup_{\mathbf{m}_1 \in M_1} d(\mathbf{m}_1,M_2)
,\sup_{\mathbf{m}_2 \in M_2} d(\mathbf{m}_2,M_1) \},
\end{displaymath}
where, for $S \in \mathcal{M}$,
\begin{displaymath}
d(\mathbf{y},S) = \inf_{\mathbf{m} \in S} \|\mathbf{y}-\mathbf{m}\|.
\end{displaymath}
\end{definition}

Next, we review the concept of  recession cones~\cite[p.~8]{NonlinearMOO3_book} and
the fundamental definitions of contingent cones \cite[p.~121]{SVA_book} and proximal normal cones \cite[p.~170]{OptConTh5_book}.\\

\begin{definition}[Recession cone]
For a nonempty subset  $S$ of $\mathbf{R}^p$, the extended recession
cone $S^+$ is defined by:
$$
S^+ = \{\mathbf{y} \in \mathbf{R}^p \ \mathrm{s.t.} \ \exists h_k \rightarrow 0^+, \ \exists \mathbf{y}_k \in S, \  h_k \mathbf{y}_k \rightarrow \mathbf{y} \}.
$$
\end{definition}

\begin{definition}[Contingent cone]
Let $S$ be a nonempty subset of some normed linear space $Z$ and let $\mathbf{z} \in S$. Then, the contingent cone
$T_S(\mathbf{z})$ to $S$ at $\mathbf{z}$ is defined by:
$$
T_S(\mathbf{z}) = \{\mathbf{v} \in Z \ \mathrm{s.t.} \ \exists h_k \rightarrow 0^+, \ \exists \mathbf{v}_k \rightarrow \mathbf{v}, \ \mathbf{z} + h_k \mathbf{v}_k \in S \}.
$$
\end{definition}

\begin{definition} [Proximal normal cone]
Let $S$ be a nonempty subset of some  Hilbert space $Z$ with inner product $\langle \cdot,\cdot \rangle$ and let $\mathbf{z} \in S$. Then, the proximal
normal cone $N_S(\mathbf{z})$ to $S$ at $\mathbf{z}$ is defined by:
$$
N_S(\mathbf{z}) = \{\mathbf{v} \in Z: \ \exists  M > 0 \ \mathrm{s.t.} \ \langle \mathbf{v}, \mathbf{\overline{z}} - \mathbf{z}\rangle \leq M \|\mathbf{\overline{z}} - \mathbf{z}\| \ \mathrm{for} \ \mathrm{all} \ \mathbf{\overline{z}} \in S \}.
$$
\end{definition}
Note that we use the notation $N_S$ for the proximal normal cone instead of the traditional notation $N^P_S$ to avoid any possible confusion with the ordering cone $P$.\\

Using contingent cones and proximal normal cones, we can respectively introduce the concepts of contingent derivative \cite[p.~181]{SVA_book} and coderivative \cite{Mordukhovich93_moo} for set-valued maps. \\

\begin{definition} [Contingent derivative of set-valued maps]
The contingent derivative  of $F$ at $(\mathbf{x},\mathbf{y})$
is the set-valued map from $X$ to $\mathbf{R}^p$ defined by
$$ \mathrm{Graph}(DF(\mathbf{x},\mathbf{y})) = T_{\mathrm{Graph}(F)}(\mathbf{x},\mathbf{y}),$$
where $T$ is the contingent cone of $\mathrm{Graph}(F)$ at $(\mathbf{x},\mathbf{y})$, as defined above.
\end{definition} \\ \\
Note that the contingent derivative $DF(\mathbf{x},\mathbf{y})$ is a closed-valued map.\\

\begin{definition} [Coderivative of set-valued maps]
The coderivative $DF^*(\mathbf{x},\mathbf{y})$ of $F$ at $(\mathbf{x},\mathbf{y})$
is the set-valued map from $\mathbf{R}^p$ to $X$ defined by:
$$DF^*(\mathbf{x},\mathbf{y})(\mathbf{w^*}) = \{\mathbf{v^*} \in X, \ (\mathbf{v^*},-\mathbf{w^*}) \in N_{\mathrm{Graph}(F)}(\mathbf{x},\mathbf{y})\},$$
where $N_{\mathrm{Graph}(F)}$ is the proximal normal cone of $\mathrm{Graph}(F)$ at $(\mathbf{x},\mathbf{y})$, as defined above.
\end{definition} \\

Finally, we will also need the notion of Lipschitz~\cite[p.~41]{SVA_book} and outersemicontinuous~\cite[p.~152]{SVA3_book} set-valued maps.\\

\begin{definition} [Lipschitz set-valued maps]
The set-valued map $F$ is said to be Lipschitz around $\mathbf{x} \in X$  if there exists a positive constant $l$ and a neighborhood $\mathcal{O}$ of $\mathbf{x}$ such that $$\forall \mathbf{x_1},\mathbf{x_2} \in \mathcal{O}, \ F(\mathbf{x_1}) \subset F(\mathbf{x_2}) + l\|\mathbf{x_1}-\mathbf{x_2}\|\mathbf{B}.$$
\end{definition}
For the definition of outersemicontinuity, we assume that $X$ is finite dimensional. For our purposes, we will have
$X = \mathbf{R}^{n+1}$.\\

\begin{definition} [Outer semicontinuous set-valued maps]
The set-valued map $F$ is said to be outer semicontinuous
at $\mathbf{x}$ if
$$ F(\mathbf{x}) \supset \limsup_{\mathbf{x'}\rightarrow \mathbf{x}} F(\mathbf{x'}) := \{\mathbf{y} \ \mathrm{s.t.} \ \exists \mathbf{x}_k
\rightarrow \mathbf{x}, \ \exists \mathbf{y}_k
\rightarrow \mathbf{y} \ \mathrm{with} \ \mathbf{y}_k \in F(\mathbf{x}_k)\}.$$
\end{definition}

\section{Multiobjective Optimization} \label{s:moogeneral}

In this section, we first discuss the concept of optimality in multiobjective
optimization and introduce the notion of generalized Pareto optimal elements or
minimal elements for a nonempty subset $S$ of $\mathbf{R}^p$. We next provide several properties related to these elements. We finally conclude by studying the Lipschitzian property of the map $E: 2^{\mathbf{R}^p} \rightarrow 2^{\mathbf{R}^p}$ which associates to each nonempty subset $S$ of $\mathbf{R}^p$ the set of minimal elements of $S$.

\subsection{Optimality in multiobjective optimization} \label{ss:opti}

For an optimization problem with a $p$-dimensional vector-valued objective function, the definition
of an optimal solution requires the comparison of any two objective vectors $\mathbf{y_1},\mathbf{y_2}$ in the objective space, which is the set of all possible values that can
be taken by the vector-valued objective function. This comparison is provided by a binary
relation expressing the preferences of the decision maker. In this paper,
we consider the  binary relation defined in terms of an ordering cone
$P \subset \mathbf{R}^p$, which is defined as a nonempty pointed
convex cone containing the origin \cite{Yu74_moo}. We will additionally assume that $P$ is closed with a nonempty interior, i.e., $\mathrm{int}(P) \neq \emptyset$.\\
\begin{definition} \label{def:binaryRelation}
Let $\mathbf{y_1},\mathbf{y_2} \in \mathbf{R}^p$. Then, $\mathbf{y_1} \preceq \mathbf{y_2}$  if and only if
$\mathbf{y_2} \in \mathbf{y_1} + P$.
\end{definition}\\

The binary relation in Definition \ref{def:binaryRelation}  yields  the definition
of generalized Pareto optimality.\\
\begin{definition} \label{def:minele}
Let $S$ be a nonempty subset of $\mathbf{R}^p$. An element $\mathbf{y_1} \in S$ is said to
be a generalized Pareto optimal element or a minimal element of $S$ if and only
if there is no $\mathbf{y_2} \in S \ (\mathbf{y_2} \neq \mathbf{y_1})$ such that
$\mathbf{y_1} \in \mathbf{y_2} + P,$ or equivalently, if and only if  there is no $\mathbf{y_2}$ such that
$\mathbf{y_1 }\in \mathbf{y_2} + P\backslash\{\mathbf{0}\}$. The set of minimal elements of
$S$ is called the minimal element set and is denoted by $\mathcal{E}(S,P)$.
\end{definition}\\ \\
Definition \ref{def:minele} can be rewritten as follows:$$ \mathbf{y} \in \mathcal{E}(S,P) \Leftrightarrow  (\mathbf{y} -P) \cap S = \{\mathbf{y} \}.$$\\
It is possible to derive a necessary condition for generalized Pareto optimality in terms of
contingent cones.\\

\begin{lemma} \label{lem:final}
Let $S$ be a nonempty subset of $\mathbf{R}^p$ and $\mathbf{y} \in \mathcal{E}(S,P)$. Then,
\begin{equation} \label{eq:nec_condition}
T_{S+P}(\mathbf{y}) \cap -\mathrm{int}(P) = \emptyset.
\end{equation}
\end{lemma}
\begin{proof}
Assume that there exists $\mathbf{d} \in \mathrm{int}(P)$ such that $-\mathbf{d} \in T_{S+P}(\mathbf{y})$. Then, there  exist sequences $h_k \rightarrow 0^+ \ \mathrm{and} \  \mathbf{d}_k \rightarrow -\mathbf{d} \ \mathrm{such \ that} \ \forall k, \ \mathbf{y} + h_k \mathbf{d}_k \in S+P.$
For large enough $k$, $-\mathbf{d}_k \in \mathrm{int}(P)$, therefore $-h_k \mathbf{d}_k \in \mathrm{int}(P)$. Hence, $\mathbf{y} \in S + P + \mathrm{int}(P) \subset S + \mathrm{int}(P)$, which contradicts the fact that $\mathbf{y} \in \mathcal{E}(S,P)$.
\end{proof}\\

Condition (\ref{eq:nec_condition}) yields the definition of the  following stronger optimality notion \cite[pp.~108-109]{SVA2_book}. \\

\begin{definition}[Generalized proper Pareto optimality] Let $S$ be a nonempty subset of $\mathbf{R}^p$.
An element $\mathbf{y} \in S$ is said to be a generalized properly Pareto optimal element or a properly minimal element of $S$,
if and only if $\mathbf{y}$ is a minimal element of  $S$ and the zero element is a minimal element of the contingent cone
$T_{S+P}(\mathbf{y})$, i.e., \begin{equation} \label{eq:propminelem} T_{S+P}(\mathbf{y}) \cap -P = \{\mathbf{0}\}.\end{equation} The set of properly minimal elements of $S$ is called the properly minimal element set and is denoted  $\mathcal{PE}(S,P)\subset \mathcal{E}(S,P).$\\
\end{definition}

\subsection{Some properties of the minimal element set}

A critical role in this paper is played by  the external stability or domination property \cite[pp.~59-66]{Tanino88_moo}.\\

\begin{definition}[External stability]
A nonempty subset $S$ of $\mathbf{R}^p$ is said to be externally stable if and only
if $$ S \subset \mathcal{E}(S,P) + P.$$
\end{definition}
An immediate consequence of the external stability property is that $S + P  = \mathcal{E}(S,P) + P$. A sufficient condition for a nonempty closed set $S$ to be externally stable is given in Proposition~\ref{prop:exter}. Note that this condition is also sufficient to guarantee the existence of minimal elements.\\

\begin{proposition}[Theorem 3.2.10, \emph{\emph{\cite[p.62]{Tanino88_moo}}}] \label{prop:exter}
Let $S$ be a nonempty closed subset of $\mathbf{R}^p$. If $S$ is $P$-bounded \emph{\cite[p.~52]{Tanino88_moo}}, i.e., $S^+ \cap - P = \{\mathbf{0}\},$ then $S$ is externally stable.\\
\end{proposition}

\begin{corollary} \label{cor:extstabcompactset}
Let $K \in \mathcal{K}$. Then $K$ is externally stable.
\end{corollary}

\begin{proof}
$K$ is a compact set, hence $K^+ = \{\mathbf{0}\}$ (Lemma 3.2.1, \cite[p.52]{Tanino88_moo}).
\end{proof} \\

Having defined the notion of minimal elements, we can now introduce the concept of generalized contingent
epiderivative for set-valued maps \cite{Chen98_moo,Jahn97_moo}, which derives from the concept of contingent derivative
for set-valued maps.\\

\begin{definition} [Generalized contingent epiderivative of set-valued maps]
The generalized contingent epiderivative  of $F$ at $(\mathbf{x},\mathbf{y}) \in \mathrm{Graph}(F)$ is the set-valued map from $X$ to $\mathbf{R}^p$ defined by
$$ \forall \mathbf{v} \in  X, \ D_\uparrow F((\mathbf{x},\mathbf{y});\mathbf{v}) = \mathcal{E}(DF_\uparrow((\mathbf{x},\mathbf{y});\mathbf{v}) ,P),
 $$
where $F_\uparrow$ is the set-valued map from $X$ to $\mathbf{R}^p$ defined by $\forall \mathbf{x} \in  X, \ F_\uparrow(\mathbf{x})
= F(\mathbf{x}) + P$. When $DF_\uparrow((\mathbf{x},\mathbf{y});\mathbf{v})$ is empty, we set $D_\uparrow F((\mathbf{x},\mathbf{y});\mathbf{v}) = \emptyset.$
\end{definition} \\  \\
It is possible that the set $DF_\uparrow((\mathbf{x},\mathbf{y});\mathbf{v})$ does not have any minimal
element. In such a case, $D_\uparrow F((\mathbf{x},\mathbf{y});\mathbf{v})$ is just
the empty set.\\

We conclude this section by a lemma that will be used in \S \ref{s:moo}.\\

\begin{lemma}[\cite{Guigue09}] \label{lem:dp4}
Let $K_1, K_2 \in \mathcal{K}$ satisfying $K_1 \subset K_2$ and $K_2 \subset K_1 + P$. Then
$\mathcal{E}(K_1,P) = \mathcal{E}(K_2,P)$.
\end{lemma}

\subsection{Lipschitzian properties} \label{ss:lip}

In this section, we analyze the map $E(\cdot): 2^{\mathbf{R}^p} \rightarrow 2^{\mathbf{R}^p}$ which
associates to each nonempty subset $S$ of $\mathbf{R}^p$ its minimal element set $\mathcal{E}(S,P)$: $$\forall S \in 2^{\mathbf{R}^p}, \ E(S) = \mathcal{E}(S,P).$$ More precisely, given $K_1, K_2 \in \mathcal{K}$, we show that, under suitable conditions, there exists $M \geq 0$ such that
\begin{equation} \label{eq:lip1}
\mathcal{H}(E(K_1),E(K_2))  \leq M \mathcal{H}(K_1,K_2).
\end{equation}

First, we record some elementary consequences of our standing assumption that $P$ is a pointed closed convex cone
with nonempty interior.\\

\begin{lemma} \label{lem:bounding3}
Let $l>0$, and define the set
$P_l = \{\mathbf{x} \in \mathbf{R}^p, \ \mathbf{B}(\mathbf{x},l) \subset P \}$. Then $P_l$ satisfies the following properties:\\
\begin{enumerate}
  \item $P_l$ is nonempty.
  \item $P_l$ is closed.
  \item $P_l$ is convex.\\
  \newcounter{enumii_saved}
  \setcounter{enumii_saved}{\value{enumi}}
\end{enumerate}
Therefore, the origin of $\mathbf{R}^p$ has a unique nearest point $\mathbf{d}_l$ in $P_l$, and we have:\\
\begin{enumerate}
  \setcounter{enumi}{\value{enumii_saved}}
  \item $\mathbf{d}_l+P \subset P_l$.
  \item $\|\mathbf{d}_l\| > l$.
  \item $\|\mathbf{d}_l\| = l \|\mathbf{d}_1\|$.\\
  \newcounter{enumii2_saved}
  \setcounter{enumii2_saved}{\value{enumi}}
\end{enumerate}
Finally,\\
\begin{enumerate}
  \setcounter{enumi}{\value{enumii2_saved}}
  \item $\forall \mathbf{x} \in \mathbf{d}_l + P, \ \mathbf{B}(\mathbf{x},l) \subset P\backslash \{\mathbf{0}\}$.\\
\end{enumerate}
\end{lemma}
\begin{proof}
\begin{enumerate}
\item This follows from the fact that $P$ is a cone with a nonempty interior.
\item Let $\mathbf{x}_k$ be a sequence in $P_l$ converging to $\mathbf{x}$. Let $\mathbf{y} \in \mathbf{R}^p$ such that
$\|\mathbf{y}-\mathbf{x}\| \leq l$. We can write $\|\mathbf{y}-\mathbf{x}\| = \|(\mathbf{y}-\mathbf{x}+\mathbf{x}_k)-\mathbf{x}_k\| \leq l$. As $\mathbf{x}_k \in P_l$, it follows that $\mathbf{y}-\mathbf{x}+\mathbf{x}_k \in P$. Taking the limit and knowing that $P$ is closed yields
$\mathbf{y} \in P$.
\item This follows directly from the convexity of the norm $\|\cdot\|$.
\item Let $\mathbf{d}\in P$ and $\mathbf{y} \in \mathbf{R}^p$ such that $\|\mathbf{y}-(\mathbf{d}_l+\mathbf{d})\| \leq l$. We can write $\|\mathbf{y}-(\mathbf{d}_l+\mathbf{d}) \| = \|(\mathbf{y}-\mathbf{d})-\mathbf{d}_l\| \leq l$. As $\mathbf{d}_l \in P_l$, it follows that $\mathbf{y}-\mathbf{d} \in P$. As $P$ is convex, we get $\mathbf{y} = \mathbf{y}-\mathbf{d} + \mathbf{d} \in P+P \subset P$.
\item Since $P$ is pointed, $\mathbf{0} \notin \mathrm{int}(P)$. Therefore, $\|\mathbf{x}\| \geq l$ for every
$\mathbf{x} \in P_l$; in particular, $\|\mathbf{d}_l\| \geq l.$\\
Pointedness also requires strict inequality here. To see why, suppose $\|\mathbf{d}_l\| = l.$ Consider an arbitrary $\mathbf{x}$  such that
    $\langle \mathbf{d}_l,\mathbf{x} \rangle > 0$, and define $\lambda = \langle \mathbf{d}_l,\mathbf{x} \rangle / \|\mathbf{x}\|^2$. Then $$\|\lambda \mathbf{x} - \mathbf{d}_l\|^2 = \lambda^2 \|\mathbf{x}\|^2 - 2 \lambda \langle \mathbf{d}_l,\mathbf{x} \rangle +l^2 = l^2-\langle \mathbf{d}_l,\mathbf{x}\rangle^2 / \|\mathbf{x}\|^2.$$  Hence,
    $\|\lambda \mathbf{x} - \mathbf{d}_l\| < l$ and $\lambda \mathbf{x} \in P$. As $\lambda > 0$, it follows that $\mathbf{x} \in P$. This shows that every $\mathbf{x}$ in the open half-space defined by $\langle \mathbf{d}_l,\mathbf{x} \rangle > 0$ lies in $P$. $P$ being closed, we get that $P$ contains a closed half-space, which contradicts that the fact $P$ is pointed.
\item It can easily be seen that $P_l = l P_1$. Therefore, $\|\mathbf{d}_l\| = \inf \{ \|\mathbf{x}\|, \ \mathbf{x} \in P_l \} = l \inf \{ \|\mathbf{x}\|, \ \mathbf{x} \in P_1 \} = l \|\mathbf{d}_1\| = l \mu(P)$.
\item This follows directly from (4) and (5).\\
\end{enumerate}
\end{proof}

As $P$ has a nonempty interior, there exists  a closed ordering cone with nonempty interior $C$ such that
$P \subset \mathrm{int}(C)\cup\{\mathbf{0}\}$. We can therefore define $\mathcal{K}(C,P)$ as the set
of all nonempty compact subsets $K$ of $\mathbf{R}^p$ such that
\begin{equation} \label{eq:cond3}
\mathcal{E}(K,P) = \mathcal{E}(K,C).
\end{equation}
The proof of (\ref{eq:lip1}), completed in Proposition \ref{prop:lips1},  requires two technical lemmas.  Lemma \ref{lem:bounding2}, which uses Lemma~\ref{lem:bounding1}, states the boundedness of a set introduced in the proof of Proposition~\ref{prop:lips1}. Hereinafter, the complement of a set $S \subset \mathbf{R}^p$ is denoted $S^c$.\\

\begin{lemma} \label{lem:bounding1}
Let  $\alpha(C,P) = \inf \{ \|\mathbf{x}-\mathbf{y}\|, \ \mathbf{x} \in \widehat{P}, \ \mathbf{y} \in  (\mathrm{int}(C))^c\}$, where $\widehat{P} = \{\mathbf{d} \in P, \ \|\mathbf{d}\| = 1 \}$. Then $\alpha(C,P) > 0$.
\end{lemma}
\begin{proof}
Assume for contradiction that $\alpha(C,P) = 0$. Then there exist two sequences  $\mathbf{x}_k$ in $\widehat{P}$ and  $\mathbf{y}_k$ in $(\mathrm{int}(C))^c$
such that $\lim \|\mathbf{x}_k-\mathbf{y}_k \| =0$. As $\widehat{P}$ is compact, by passing to a subsequence
if necessary, we can assume that $\mathbf{x}_k$ converges to some $\mathbf{x} \in \widehat{P}$. Therefore, the sequence $\mathbf{y}_k$ is bounded. By passing to a subsequence if necessary, we can assume that
$\mathbf{y}_k$ converges to some  $\mathbf{y}$. Clearly, $\mathbf{y} = \mathbf{x}$; also,
$\mathbf{y} \in (\mathrm{int}(C))^c$ as $(\mathrm{int}(C))^c$ is closed. By assumption,  $P \subset \mathrm{int}(C)\cup\{\mathbf{0}\}$. It follows that $\mathbf{y} \in (\mathrm{int}(C)\cup\{\mathbf{0}\}) \cap (\mathrm{int}(C))^c$, which implies
$\mathbf{y}=\mathbf{0}$ and therefore $\mathbf{x}=\mathbf{0}$, which  contradicts $\mathbf{x} \in \widehat{P}$.
\end{proof}\\

\begin{lemma} \label{lem:bounding2}
Let $\mathbf{y} \in (\mathrm{int}(C))^c$. Then,  $\forall \mathbf{x} \in (\mathbf{y}+P) \cap (\mathrm{int}(C))^c, \ \|\mathbf{x}\| \leq  \alpha'(C,P) \|\mathbf{y}\|$, where $\alpha'(C,P) = (1+1 / \alpha (C,P)) > 1$.
\end{lemma}
\begin{proof}
Let $\mathbf{x} \in (\mathbf{y}+P) \cap (\mathrm{int}(C))^c$. Choose $\mathbf{d} \in P$ such that $\mathbf{x} = \mathbf{y}+\mathbf{d}$. If $\mathbf{d}=\mathbf{0}$, then $\mathbf{x} = \mathbf{y}$ and $\|\mathbf{x}\| = \| \mathbf{y}\| \leq \alpha'(C,P) \| \mathbf{y}\|$ as $\alpha'(C,P) > 1$. Otherwise, we can write $\mathbf{x} / \|\mathbf{d}\| - \mathbf{d} / \|\mathbf{d}\| = \mathbf{y} / \|\mathbf{d}\|$. We have
$ \mathbf{d} / \|\mathbf{d}\| \in \widehat{P}$, $\mathbf{x} / \|\mathbf{d}\| \in (\mathrm{int}(C))^c$, therefore from Lemma \ref{lem:bounding1}, $\|\mathbf{y}\|  / \|\mathbf{d}\| \geq \alpha (C,P)$. Finally, $\| \mathbf{x}\| \leq \|\mathbf{y}\| + \| \mathbf{d}\| \leq (1+1 / \alpha (C,P)) \|\mathbf{y}\|$.
\end{proof}\\

\begin{proposition} \label{prop:lips1}
There exists a constant $M(C,P) \geq 0$ such that $\forall K_1,K_2 \in \mathcal{K}(C,P)$,
\begin{displaymath}
 \mathcal{H}(\mathcal{E}(K_1,P),\mathcal{E}(K_2,P)) \leq M(C,P)  \ \mathcal{H}(K_1,K_2).
\end{displaymath}
\end{proposition}
\begin{proof}
If $\mathcal{H}(K_1,K_2) = 0$, then $K_1 = K_2$ and the result is obvious. Therefore, assume that
 $\mathcal{H}(K_1,K_2) > 0$ and let $l= \mathcal{H}(K_1,K_2)$. Let $\mathbf{y}_1 \in \mathcal{E}(K_1,P)$. Then by definition of $l$, there exists $\mathbf{k}_2 \in K_2$ such that $\|\mathbf{y}_1-\mathbf{k}_2\| \leq l.$ From Corollary \ref{cor:extstabcompactset}, $K_2$ is externally stable, hence there exists $\mathbf{y}_2 \in \mathcal{E}(K_2,P)$ such that $\mathbf{k}_2 \in \mathbf{y}_2 + P$. From Lemma \ref{lem:bounding3}(7), there exists $\mathbf{d}_l \in C$ such that $\forall \mathbf{x} \in \mathbf{d}_l + C, \ B(\mathbf{x},l) \subset C\backslash \{\mathbf{0}\}$. We prove now that $\mathbf{y}_1-\mathbf{y}_2- \mathbf{d}_l \in (\mathrm{int}(C))^c$. Assume that $\mathbf{y}_1-\mathbf{y}_2 \in \mathbf{d}_l + C$, then $B(\mathbf{y}_1-\mathbf{y}_2,l) \in C \backslash \{\mathbf{0}\} $,
and therefore $B(\mathbf{y}_2,l) \subset \mathbf{y}_1 - C \backslash \{\mathbf{0}\}$. By definition of $l$ again, there exists $\mathbf{k}_1 \in K_1$
such that $\mathbf{k}_1 \in B(\mathbf{y}_2,l)$. Hence, $\mathbf{y}_1 \in \mathbf{k}_1 + C \backslash \{\mathbf{0}\}$. But, $\mathbf{y}_1 \in \mathcal{E}(K_1,P)$ and by
(\ref{eq:cond3}), $\mathcal{E}(K_1,P) = \mathcal{E}(K_1,C)$. Therefore, we obtain a contradiction. Finally, $\mathbf{y}_1-\mathbf{y}_2 \notin \mathbf{d}_l + C$. Hence, $\mathbf{y}_1-\mathbf{y}_2-\mathbf{d}_l \in (\mathrm{int}(C))^c$. Now, we want to use
Lemma \ref{lem:bounding2} with $\mathbf{y} = \mathbf{y}_1-\mathbf{k}_2-\mathbf{d}_l$ and $\mathbf{x} = \mathbf{y}_1-\mathbf{y}_2-\mathbf{d}_l$. As $\mathbf{x} \in (\mathrm{int}(C))^c$ and $\mathbf{x} -\mathbf{y} = \mathbf{k}_2 - \mathbf{y}_2 \in P$, we only need to check that $\mathbf{y} \in  (\mathrm{int}(C))^c$. If $ \mathbf{y}_1-\mathbf{k}_2  \in \mathbf{d}_l + C$, then
by definition of $\mathbf{d}_l$, $\| \mathbf{y}_1-\mathbf{k}_2 \| > l$, which is a contradiction. Hence, $\mathbf{y}_1-\mathbf{k}_2 \notin \mathbf{d}_l + C$, or $\mathbf{y}_1-\mathbf{k}_2 - \mathbf{d}_l \in (\mathrm{int}(C))^c$. From Lemma~\ref{lem:bounding2}, it follows that $$ \|\mathbf{y}_1-\mathbf{y}_2-\mathbf{d}_l\| \leq \alpha'(C,P) \|\mathbf{y}_1-\mathbf{k}_2-\mathbf{d}_l\|.$$
Hence, $$\|\mathbf{y}_1-\mathbf{y}_2\| \leq (1+\alpha'(C,P)) \|\mathbf{d}_l\| + \alpha'(C,P) \|\mathbf{y}_1-\mathbf{k}_2\|.$$
From Lemma \ref{lem:bounding3}(6), we have $\|\mathbf{d}_l\| = l \mu(C)$. Recalling that $\|\mathbf{y}_1-\mathbf{k}_2\| \leq l$, and defining
$M(C,P) = (1+\alpha'(C,P)) \mu(C) + \alpha'(C,P)$, we get $$\|\mathbf{y}_1-\mathbf{y}_2\| \leq M(C,P) \ l,$$
which shows that $$ \sup_{\mathbf{y}_1 \in \mathcal{E}(K_1,P)} d(\mathbf{y}_1,\mathcal{E}(K_2,P)) \leq M(C,P) \ \mathcal{H}(K_1,K_2).$$
Interchanging the role of $K_1$ and $K_2$ finally yields $$ \mathcal{H}(\mathcal{E}(K_1,P),\mathcal{E}(K_2,P)) \leq M(C,P)  \ \mathcal{H}(K_1,K_2).$$
\end{proof}


We conclude this section by proving in Proposition \ref{prop:last} that  (\ref{eq:cond3}) implies that $\mathcal{E}(K,P) = \mathcal{PE}(K,P).$ In other words, a set $K \in \mathcal{K}(C,P)$ only contains
properly minimal elements.  \\


\begin{proposition} \label{prop:last}
If $K \in \mathcal{K}(C,P)$, then $\mathcal{E}(K,P) = \mathcal{PE}(K,P)$.
\end{proposition}
\begin{proof}
By definition, we have $\mathcal{PE}(K,P) \subset \mathcal{E}(K,P)$. Let $\mathbf{y} \in \mathcal{E}(K,P)$. By definition of $\mathcal{K}(C,P)$, it follows that $\mathbf{y} \in \mathcal{E}(K,C)$. Assume that $\mathbf{y} \notin \mathcal{PE}(K,P)$. Then there exists $\mathbf{d} \in P, \ \mathbf{d} \neq \mathbf{0}$, such that $-\mathbf{d} \in T_{K+P}(\mathbf{y}).$ Now, as $P \subset \mathrm{int}(C)\cup\{\mathbf{0}\}$ and $\mathbf{d} \neq \mathbf{0}$, it follows that $\mathbf{d} \in \mathrm{int}(C)$. As $K+P \subset K+C$,  $T_{K+P}(\mathbf{y}) \subset T_{K+C}(\mathbf{y})$. Therefore, $-\mathbf{d} \in T_{K+C}(\mathbf{y}) \cap -\mathrm{int}(C)$, which, from Lemma \ref{lem:final}, contradicts the fact that $\mathbf{y} \in \mathcal{E}(K,C)$.
\end{proof}

\section{A multiobjective dynamic programming equation} \label{s:moo}
In this section, we state the multiobjective dynamic programming equation satisfied by the set-valued return function $V$. For this purpose, we need to introduce some additional notation. Let $t \in I$, $ \tau \in (0, T-t]$, $\mathbf{x} \in \mathbf{R}^n$, and define
the bounded set
\begin{displaymath}
\widetilde{Y}(\tau,t,\mathbf{x}) =
\bigg \{ \mathbf{J}(t,t+\tau,\mathbf{x},\mathbf{u}(\cdot)) +
V(t+\tau,  \mathbf{x}(t+\tau;t,\mathbf{x},\mathbf{u}(\cdot))), \mathbf{u}(\cdot) \in \mathcal{U}
\bigg \}.
\end{displaymath}

\begin{proposition} \label{th:dp}
For each $t \in I$, $ \tau \in (0, T-t]$, $\mathbf{x} \in \mathbf{R}^n$, the  multiobjective
dynamic programming equation for \emph{(MOC)} is:
\begin{displaymath}
V(t,\mathbf{x}) = \mathcal{E}(\mathrm{cl}(\widetilde{Y}(\tau,t,\mathbf{x})),P),
\end{displaymath}
or, using the definition of $\widetilde{Y}(\tau,t,\mathbf{x})$ above,
\begin{equation} \label{eq:dpp}
V(t,\mathbf{x}) = \mathcal{E} \bigg( \mathrm{cl}\bigg( \bigg \{ \mathbf{J}(t,t+\tau,\mathbf{x},\mathbf{u}(\cdot)) +
V(t+\tau,  \mathbf{x}(t+\tau;t,\mathbf{x},\mathbf{u}(\cdot))), \mathbf{u}(\cdot) \in \mathcal{U}
\bigg \}\bigg),P\bigg).
\end{equation}
\end{proposition}
When $\tau = T-t$, using the fact that $V(T,\cdot) = \{\mathbf{0}\}$, it can be seen that (\ref{eq:dpp}) reduces to the definition
of $V$. Therefore, assume that $\tau < T-t$. We prove Proposition \ref{th:dp} at the end of this section using the following three lemmas.
\newline
\begin{lemma} \label{lem:dp1}
$\widetilde{Y}(\tau,t,\mathbf{x}) \subset
\mathrm{cl}(Y(t,\mathbf{x}))$.
\end{lemma}
\begin{proof} Let $\widetilde{\mathbf{y}} \in
\widetilde{Y}(\tau,t,\mathbf{x}) $ and $\epsilon > 0$. Then there
exist $\mathbf{u}(\cdot) \in \mathcal{U}$ and
$\mathbf{y} \in Y(t+\tau,  \mathbf{x}(t+\tau;t,\mathbf{x},\mathbf{u}(\cdot)))$ such that
\begin{equation} \label{eq:lemdp1}
\| \widetilde{\mathbf{y}} -
\mathbf{J}(t,t+\tau,\mathbf{x},\mathbf{u}(\cdot))-
\mathbf{y} \| \leq \epsilon.
\end{equation}
Moreover, we have $\mathbf{y} = \mathbf{J}(t+\tau,T,\mathbf{x}(t+\tau;t,\mathbf{x},\mathbf{u}(\cdot)),\mathbf{\widehat{u}}(\cdot))$
for some $\mathbf{\widehat{u}}(\cdot) \in \mathcal{U}$. Define
the new control $\widetilde{\mathbf{u}}(\cdot) \in \mathcal{U}$ as
\begin{displaymath}
\mathbf{\widetilde{u}}(s) = \left\{ \begin{array}{ll}
            \mathbf{u}(s), & t \leq s \leq t+\tau ,\\
            \mathbf{\widehat{u}}(s), & t+\tau< s \leq T.\\
             \end{array} \right.
\end{displaymath}
Observe that
$$ \mathbf{J}(t,t+\tau,\mathbf{x},\mathbf{u}(\cdot)) = \displaystyle \int_{t}^{t+\tau} \mathbf{L}(\mathbf{x}(s;t,\mathbf{x},\mathbf{\widetilde{u}}(\cdot)),\mathbf{\widetilde{u}}(s)) \ \mathrm{d}s,$$
and
\begin{displaymath}
\begin{array}{lcl}
\mathbf{y} & = &  \mathbf{J}(t+\tau,T,\mathbf{x}(t+\tau;t,\mathbf{x},\mathbf{u}(\cdot)),\widehat{\mathbf{u}}(\cdot))\\
 & = & \displaystyle \int_{t+\tau}^{T} \mathbf{L}(\mathbf{x}(s;t+\tau,\mathbf{x}(t+\tau;t,\mathbf{x},\mathbf{u}(\cdot)),\mathbf{\widehat{u}}(\cdot)),\mathbf{\widehat{u}}(s)) \ \mathrm{d}s\\
 & = & \displaystyle \int_{t+\tau}^{T} \mathbf{L}(\mathbf{x}(s;t+\tau,\mathbf{x}(t+\tau;t,\mathbf{x},\mathbf{\widetilde{u}}(\cdot)),\mathbf{\widetilde{u}}(\cdot)),\mathbf{\widetilde{u}}(s)) \ \mathrm{d}s.\\
             \end{array}
\end{displaymath}
Hence, $$ \mathbf{J}(t,t+\tau,\mathbf{x},\mathbf{u}(\cdot))  +  \mathbf{y} = \mathbf{J}(t,T,\mathbf{x},\mathbf{\widetilde{u}}(\cdot)) \in Y(t,\mathbf{x}).$$
Since $\epsilon > 0$ is arbitrary, (\ref{eq:lemdp1}) implies that $\widetilde{Y}(\tau,t,\mathbf{x}) \subset
\mathrm{cl}(Y(t,\mathbf{x}))$.  \end{proof}
\newline
\begin{lemma} \label{lem:dp2}
$Y(t,\mathbf{x}) \subset \widetilde{Y}(\tau,t,\mathbf{x})+P
$.
\end{lemma}
\begin{proof} Let $\mathbf{y} \in
Y(t,\mathbf{x})$. We can write
$\mathbf{y} = \mathbf{J}(t,t+\tau,\mathbf{x},\mathbf{u}(\cdot)) +
\mathbf{\widetilde{y}}$ with $\widetilde{\mathbf{y}} \in Y(t+\tau,\mathbf{x}(t+\tau;t,\mathbf{x},\mathbf{u}(\cdot)))$.
As $\mathrm{cl}(Y(t+\tau,\mathbf{x}(t+\tau;t,\mathbf{x},\mathbf{u}(\cdot))))$ is externally
stable (Corollary \ref{cor:extstabcompactset}), there exist $\mathbf{\widetilde{y}}^* \in
V(t+\tau,\mathbf{x}(t+\tau;t,\mathbf{x},\mathbf{u}(\cdot)))$ and $\mathbf{d} \in P$
such that $\mathbf{\widetilde{y}} = \mathbf{\widetilde{y}}^* + \mathbf{d}$. Therefore,
\begin{displaymath}
\mathbf{y} =
\mathbf{J}(t,t+\tau,\mathbf{x},\mathbf{u}(\cdot)) +
\mathbf{\widetilde{y}}^*  + \mathbf{d},
\end{displaymath}
with $\mathbf{J}(t,t+\tau,\mathbf{x},\mathbf{u}(\cdot)) +
\mathbf{\widetilde{y}}^* \in \widetilde{Y}(\tau,t,\mathbf{x})$.
Hence, $\mathbf{y} \in
\widetilde{Y}(\tau,t,\mathbf{x})+P$.  \end{proof}
\newline
\begin{lemma} \label{lem:dp3}
$\mathrm{cl}(\widetilde{Y}(\tau,t,\mathbf{x})+P) \subset \mathrm{cl}(\widetilde{Y}(\tau,t,\mathbf{x}))+P$.
\end{lemma}
\begin{proof} This is a consequence of the facts that
the set $\mathrm{cl}(\widetilde{Y}(\tau,t,\mathbf{x}))$
is bounded and $P$ is closed.  \end{proof}
\newline \newline We can now proceed with the proof of Proposition \ref{th:dp}.
\newline
\begin{proof}[Proposition \ref{th:dp}]
Apply Lemma \ref{lem:dp4} with $K_1 =
\mathrm{cl}(\widetilde{Y}(\tau,t,\mathbf{x}))$ and $K_2 =
\mathrm{cl}(Y(t,\mathbf{x}))$. The sets
$\mathrm{cl}(\widetilde{Y}(\tau,t,\mathbf{x}))$ and
$\mathrm{cl}(Y(t,\mathbf{x}))$ are compact. The inclusion $K_1
\subset K_2$ comes from Lemma \ref{lem:dp1}, while the inclusion
$K_2 \subset K_1 + P$ comes from Lemmas \ref{lem:dp2} and~\ref{lem:dp3}. \end{proof} \newline

\begin{remark}
The dynamic programming equation \emph{(\ref{eq:dpp})} obtained in Proposition~\emph{\ref{th:dp}}
is a generalization of the dynamic programming equation obtained for single objective optimal
control problems \emph{\cite{OptConTh8_book,OptConTh5_book}}. Indeed, when $p=1$ and  $P = \mathbf{R}_+$, and
using Remark \emph{\ref{rem:valuefunction1dof}}, both sets in \emph{(\ref{eq:dpp})} contain exactly one element,
so \emph{(\ref{eq:dpp})} is equivalent to
\begin{displaymath}
v(t,\mathbf{x}) =  \inf_{\mathbf{u}(\cdot) \in \mathcal{U}} \int_{t}^{t+\tau}
\mathbf{L}(\mathbf{x}(s;t,\mathbf{x},\mathbf{u}(\cdot)),\mathbf{u}(s)) \ \mathrm{d}s +
v(t+\tau,\mathbf{x}(t+\tau;t,\mathbf{x},\mathbf{u}(\cdot))).
\end{displaymath}
\end{remark}

\section{Generalized contingent solution for (MOC)} \label{s:hjb}


The notion of lower Dini solution to the Hamilton-Jacobi Equation for
finite-horizon single objective optimal control problems was introduced
in \cite[p.~454]{OptConTh5_book} (see also \cite[p.~127]{OptConTh8_book}).
In this section, using the concept of contingent derivative for set-valued maps,
we extend this notion to (MOC). We call this extended notion generalized contingent
solution for (MOC).\\

\subsection{Definition}

Our definition of generalized contingent solution for (MOC) assumes
set-valued maps of a particular type, as described in Definition \ref{def:extmap}.
This assumption will be used in Corollaries \ref{cor:uniqueness}-\ref{cor:conclusion}
to state that the set-valued return function $V$ is the unique  generalized contingent
solution for (MOC).\\

\begin{definition} \label{def:extmap}
A set-valued map $W$  from $I \times \mathbf{R}^n$ to $\mathbf{R}^p$ is said to be an extremal
element map if, for all $(t,\mathbf{x}) \in I \times \mathbf{R}^n$, for all $\mathbf{y} \in W(t,\mathbf{x})$,\\
\begin{equation} \label{eq:property1}  W(t,\mathbf{x}) \cap (\mathbf{y }- P) = \{\mathbf{y}\},\end{equation}
and
\begin{equation} \label{eq:property2}  W(t,\mathbf{x}) \cap (\mathbf{y }+ P) = \{\mathbf{y}\}.\end{equation}
\end{definition}

Let  $(\mathrm{FL})(\mathbf{x}) = \mathrm{cl}(\mathrm{co}( \{(\mathbf{f}(\mathbf{x},\mathbf{u}),\mathbf{L}(\mathbf{x},\mathbf{u})), \ \mathbf{u} \in U\})),$
where $\mathrm{co}(S)$ denotes the convex hull of the set $S$.\\

\begin{definition} \label{def:solution}
An extremal element map $W$ from $I \times \mathbf{R}^n$ to $\mathbf{R}^p$ is said to be a generalized
contingent solution for \emph{(MOC)} if:\\
\begin{itemize}
  \item For all $(t,\mathbf{x}) \in [0,T) \times \mathbf{R}^n$, for all $\mathbf{y} \in W(t,\mathbf{x}),$
there exists $(\mathbf{\overline{f}},\mathbf{\overline{L}}) \in (\mathrm{FL})(\mathbf{x})$ such that
\begin{equation} \label{eq:vfirst}
 -\mathbf{\overline{L}} \in DW_\uparrow((t,\mathbf{x},\mathbf{y});(1,\mathbf{\overline{f}})).
\end{equation}
  \item For all $(t,\mathbf{x}) \in (0,T] \times \mathbf{R}^n$, for all $\mathbf{y} \in W(t,\mathbf{x}),$
and for all $(\mathbf{f},\mathbf{L}) \in (\mathrm{FL})(\mathbf{x})$,
\begin{equation} \label{eq:vsecond}
 \mathbf{L} \in DW_\uparrow((t,\mathbf{x},\mathbf{y});(-1,-\mathbf{f})).
\end{equation}
\item For all $\mathbf{x} \in  \mathbf{R}^n$,
\begin{equation} \label{eq:tercon1}
W(T,\mathbf{x}) = \{\mathbf{0}\}.
\end{equation}
\end{itemize}
\end{definition}

\subsection{A reformulation for (\ref{eq:vfirst})}

When $W$ is Lipschitz around $(t,\mathbf{x}) \in (0,T) \times \mathbf{R}^n$ and
$\mathbf{y} \in W(t,\mathbf{x})$ is a properly minimal element, i.e.,
$\mathbf{y} \in \mathcal{PE}(W(t,\mathbf{x}),P)$, there exists a more compact formulation for (\ref{eq:vfirst}).\\

\begin{proposition} \label{prop:super}
Let $W$ be a generalized contingent solution for \emph{(MOC)}, $(t,\mathbf{x}) \in [0,T) \times \mathbf{R}^n$,
and $\mathbf{y} \in W(t,\mathbf{x})$. Assume that $W$ is Lipschitz around $(t,\mathbf{x})$ and that $\mathbf{y} \in \mathcal{PE}(W(t,\mathbf{x}),P)$.
Then we have:
\begin{equation} \label{eq:vfirst_ext}
\mathbf{0} \in \mathcal{E}\bigg(\mathrm{cl}\bigg( \bigg\{\mathbf{L} + D_\uparrow W((t,\mathbf{x},\mathbf{y});(1,\mathbf{f})), \ (\mathbf{f},\mathbf{L}) \in (\mathrm{FL})(\mathbf{x}) \bigg\} \bigg),P\bigg)  + P,
\end{equation}
\end{proposition}
We complete the proof of Proposition \ref{prop:super} later in this section. Beforehand,
we need to show that the generalized contingent epiderivatives of $W$ at $(t,\mathbf{x},\mathbf{y})$ in
the direction $(1,\mathbf{f})$, i.e., $D_\uparrow W((t,\mathbf{x},\mathbf{y});(1,\mathbf{f}))$
are nonempty. We derive this result in the same general setting as in \S \ref{s:mathprel}.\\

\begin{lemma} \label{lem:lemma1}
Let $S_1,S_2$ be nonempty subsets of $\mathbf{R}^p$. Then,\\
\begin{enumerate}
  \item $(\mathrm{cl}(S_1))^+ = S_1^+$.
  \item If $S_1 \subset S_2$, then $S_1^+ \subset  S_2^+$.
  \item Let $l \geq 0$. If $S_1 \subset S_2 + l\mathbf{B}$ and $S_2$ is $P$-bounded, then $\mathrm{cl}(S_1)$ $P$-bounded. \\
\end{enumerate}
\end{lemma}
\begin{proof}
The proof of (1) and (2) can be found in \cite[p.~9]{NonlinearMOO3_book}. To prove (3),
assume $S_1 \subset S_2 + l\mathbf{B}$ and let $\mathbf{y} \in (\mathrm{cl}(S_1))^+ \cap -P = S_1^+ \cap -P$ by (1). By definition of the recession cone, there exist  sequences $h_k \rightarrow 0^+$,
$\mathbf{\overline{y}}_k \in S_1$ such that $h_k \mathbf{\overline{y}}_k \rightarrow \mathbf{y}.$ By assumption,
we have $\mathbf{\overline{y}}_k = \mathbf{\widetilde{y}}_k + l \mathbf{y}_k $, where $\mathbf{\widetilde{y}}_k \in S_2$
and $\mathbf{y}_k \in \mathbf{B}.$ Then,
$h_k \mathbf{\widetilde{y}}_k  = h_k \mathbf{\overline{y}}_k - h_k  l \mathbf{y}_k \rightarrow \mathbf{y}.$
Hence, $\mathbf{y} \in S_2^+$ and $\mathbf{y} \in -P$, or $\mathbf{y} \in S_2^+ \cap -P$.
As $S_2$ is assumed to be $P$-bounded, it follows that
 $\mathbf{y} = \mathbf{0}$.
\end{proof} \newline

\begin{proposition} \label{prop:Lipschitz} Let $(\mathbf{x},\mathbf{y}) \in Graph(F)$. Assume that $F$ is Lipschitz
around $\mathbf{x}$ with Lipschitz constant $l$. Then, we have:\\
\begin{enumerate}
  \item $\forall \mathbf{v} \in X, \ DF((\mathbf{x},\mathbf{y});\mathbf{v})\neq \emptyset$.
  \item The set-valued map $DF(\mathbf{x},\mathbf{y})$ is Lipschitz with Lipschitz constant $l$, i.e.,
  $$\forall \mathbf{v_1}, \mathbf{v_2} \in X, \ DF(\mathbf{x},\mathbf{y})(\mathbf{v_1}) \subset DF(\mathbf{x},\mathbf{y})(\mathbf{v_2}) + l \mathbf{B}.$$
  \item $DF_\uparrow((\mathbf{x},\mathbf{y});\mathbf{0}) \subset T_{F_\uparrow(\mathbf{x})}(\mathbf{y}).$
\end{enumerate}
\end{proposition}
\begin{proof}
The proof of (1) and (2) can be found in \cite[p.~186]{SVA_book}. To prove (3), let
$\mathbf{w} \in DF_\uparrow((\mathbf{x},\mathbf{y});\mathbf{0}).$
By definition of the contingent derivative, there exist sequences
$h_k\rightarrow 0^+$, $\mathbf{v}_k\rightarrow  0$, and $\mathbf{w}_k\rightarrow
\mathbf{w}$ such that
$$ \mathbf{y} + h_k \mathbf{w}_k \in F(\mathbf{x} + h_k \mathbf{v}_k) + P. $$
Using the Lipschitz property, we get $$ \mathbf{y} + h_k \mathbf{w}_k \in F(\mathbf{x}) + l
h_k \|\mathbf{v}_k\| \mathbf{B} + P.$$
Hence,  $$ \mathbf{y} + h_k (\mathbf{w}_k - l \|\mathbf{v}_k\| \mathbf{y}_k)
\in F(\mathbf{x}) + P$$ for some sequence $\mathbf{y}_k \in \mathbf{B}.$ We have $\mathbf{w}_k - l \|\mathbf{v}_k\| \mathbf{y}_k
 \rightarrow \mathbf{w},$ which shows
that $\mathbf{w} \in T_{F_\uparrow(\mathbf{x})}(\mathbf{y})$.
\end{proof}\\ \\
The conclusion of Proposition \ref{prop:Lipschitz}(1) might fail if $F$ is not Lipschitz, but simply continuous.
Take for example the set-valued map $F$ defined from $\mathbf{R}$ to $\mathbf{R}$ by $F(x) = \{x^{1/3}\}$ and $P = \mathbf{R}_+$. Then, it is
easy to check that
$DF((0,0);v) = \emptyset \ \mathrm{when} \  v \neq 0.$\\

\begin{proposition} \label{prop:prop2}
Let $(\mathbf{x},\mathbf{y}) \in Graph(F)$. Assume that $F$ is Lipschitz
around $\mathbf{x}$ and that $\mathbf{y} \in \mathcal{PE}(F(\mathbf{x}),P)$. Then,
\begin{equation} \label{eq:properlyminimalelements}
DF_\uparrow((\mathbf{x},\mathbf{y});\mathbf{0}) \cap -P = \{\mathbf{0}\}.
\end{equation}
Moreover, $\forall \mathbf{v} \in X$, the sets $DF_\uparrow((\mathbf{x},\mathbf{y});\mathbf{v})$ and
$DF((\mathbf{x},\mathbf{y});\mathbf{v})$ are $P$-bounded.
\end{proposition}
\begin{proof}
As $F$ is Lipschitz, $F_\uparrow$ is also Lipschitz. Therefore, from
Proposition \ref{prop:Lipschitz}(1), $DF_\uparrow((\mathbf{x},\mathbf{y});\mathbf{0})\neq \emptyset$.
Now, let $\mathbf{w} \in DF_\uparrow((\mathbf{x},\mathbf{y});\mathbf{0}) \cap -P.$ From
Proposition \ref{prop:Lipschitz}(3), $DF_\uparrow((\mathbf{x},\mathbf{y});\mathbf{0}) \subset T_{F_\uparrow(\mathbf{x})}(\mathbf{y})$. Hence, $\mathbf{w} \in T_{F_\uparrow(\mathbf{x})}(\mathbf{y}) \cap -P.$
By assumption, $\mathbf{y} \in \mathcal{PE}(F(\mathbf{x}),P)$. Therefore, $\mathbf{w} =0$.\\ \\ To conclude that the sets $DF_\uparrow((\mathbf{x},\mathbf{y});\mathbf{v})$ are $P$-bounded, observe
first that the set  $DF_\uparrow((\mathbf{x},\mathbf{y});\mathbf{0})$ is $P$-bounded. Indeed,
$DF_\uparrow((\mathbf{x},\mathbf{y});\mathbf{0})$ is a closed cone, hence $DF_\uparrow((\mathbf{x},\mathbf{y});\mathbf{0})^+
 = DF_\uparrow((\mathbf{x},\mathbf{y});\mathbf{0})$, or $DF_\uparrow((\mathbf{x},\mathbf{y});\mathbf{0})^+ \cap -P =
 DF_\uparrow((\mathbf{x},\mathbf{y});\mathbf{0}) \cap -P = \{\mathbf{0}\}.$ Now, from Proposition \ref{prop:Lipschitz}(2), the set-valued map  $DF_\uparrow(\mathbf{x},\mathbf{y})$ is Lipschitz. The conclusion therefore follows  from Lemma \ref{lem:lemma1}(3). $P$-boundedness of the sets $DF((\mathbf{x},\mathbf{y});\mathbf{v})$ is readily obtained
 from the inclusion $DF((\mathbf{x},\mathbf{y});\mathbf{v})+P \subset DF_\uparrow((\mathbf{x},\mathbf{y});\mathbf{v}).$
\end{proof} \\

The assumption that $\mathbf{y}$ is a properly minimal element is essential in obtaining that
the sets $DF_\uparrow((\mathbf{x},\mathbf{y});\mathbf{v})$ are $P$-bounded. If $\mathbf{y}$ is
only assumed to be a minimal element, then, using Lemma \ref{lem:final} and
Proposition \ref{prop:Lipschitz}, instead of (\ref{eq:properlyminimalelements}), we would get
$DF_\uparrow((\mathbf{x},\mathbf{y});\mathbf{0}) \cap -\mathrm{int}(P) = \emptyset.$ Therefore,
when $\mathbf{y}$ is only a minimal element,  $DF_\uparrow((\mathbf{x},\mathbf{y});\mathbf{0})$ is not necessarily
$P$-bounded, as illustrated by the following example. Let $X = \mathbf{R}$, $p=2$, and
 $P=\mathbf{R}^2_+$. Let $S = \{ (\mathbf{y}_1,\mathbf{y}_1^2) \ \mathrm{if} \ \mathbf{y}_1 \leq 0\} \cup \
\{(\mathbf{y}_1,-\mathbf{y}_1) \ \mathrm{if} \ \mathbf{y}_1 > 0\}$. Take $(0,0) \in S$. We have that $(0,0)$ is
a minimal element. Moreover,
$T_{S+P}(0,0) =  \{ (\mathbf{w}_1,\mathbf{w}_2), \ \mathbf{w}_2 \geq 0 \} \cup \ \{(\mathbf{w}_1,\mathbf{w}_2), \ \mathbf{w}_2 \geq -\mathbf{w}_1\} $
Hence, $T_{S+P}(0,0) \cap -\mathrm{int}(P) = \emptyset$, $ T_{S+P}(0,0) \cap -P  = \{(\mathbf{w_1},0), \ \mathbf{w_1} \leq 0\}\neq \{(0,0)\},$
and therefore $(0,0)$ is not a properly minimal element. Now, define the constant set-valued map: $\forall x \in \mathbf{R}, \ F(x) = S.$
The set-valued map $F$ is obviously
Lipschitz around all $x \in \mathbf{R}$, and it is not hard to show that
$\forall x \in \mathbf{R}, \ DF_\uparrow((x,0,0);0) = T_{S+P}(0,0)$. Hence, the set
$DF_\uparrow((x,0,0);0)$ is not $P$-bounded.\\

Proposition \ref{prop:prop2} implies that the sets $D_\uparrow W((t,\mathbf{x},\mathbf{y});(1,\mathbf{f}))$
appearing in  (\ref{eq:vfirst_ext}) are nonempty. To see this, note from
Proposition \ref{prop:exter} that it suffices to show that the sets $D W_\uparrow((t,\mathbf{x},\mathbf{y});(1,\mathbf{f}))$  are $P$-bounded. Applying Proposition \ref{prop:prop2}
to $W$ achieves this.\\

The next step is to show  $P$-boundedness of the set $$\mathrm{cl}\bigg( \bigg\{\mathbf{L} +
D_\uparrow W((t,\mathbf{x},\mathbf{y});(1,\mathbf{f})), \ (\mathbf{f},\mathbf{L})  \in (\mathrm{FL})(\mathbf{x}) \bigg\} \bigg).$$ Again, we derive this result in a general setting.\\

\begin{proposition} \label{prop:prop4}
Let $(\mathbf{x},\mathbf{y}) \in Graph(F)$. Assume that $F$ is Lipschitz
around $\mathbf{x}$ and that $\mathbf{y} \in \mathcal{PE}(F(\mathbf{x}),P)$.
If $S$ is a nonempty compact subset  of $ X \times \mathbf{R}^p$, then the following set is $P$-bounded:
$$\mathrm{cl}\bigg(\displaystyle \bigcup_{(\mathbf{v},\mathbf{w}) \in S} \mathbf{w} + D_\uparrow F((\mathbf{x},\mathbf{y});\mathbf{v})\bigg).$$
\end{proposition}

\begin{proof}
As $D_\uparrow F((\mathbf{x},\mathbf{y});\mathbf{v}) \subset D F_\uparrow((\mathbf{x},\mathbf{y});\mathbf{v})$,
from Lemma \ref{lem:lemma1}(1)-(2), it is enough to show that  $\bigcup_{(\mathbf{v},\mathbf{w}) \in S}
\mathbf{w} + D F_\uparrow((\mathbf{x},\mathbf{y});\mathbf{v})$ is $P$-bounded.
Let $$\mathbf{\overline{w}} \in \bigg( \displaystyle \bigcup_{(\mathbf{v},\mathbf{w}) \in S}
\mathbf{w} + D F_\uparrow((\mathbf{x},\mathbf{y});\mathbf{v})\bigg)^+\cap -P.$$
By definition of the recession cone, there exist sequences $h_k \rightarrow 0^+$,
$\mathbf{\overline{w}}_k \in \bigcup_{(\mathbf{v},\mathbf{w}) \in S} \mathbf{w} +
D F_\uparrow((\mathbf{x},\mathbf{y});\mathbf{v})$ such
that $h_k \mathbf{\overline{w}}_k \rightarrow \overline{\mathbf{w}}.$
We have $\mathbf{\overline{w}}_k \in \mathbf{w}_k + D F_\uparrow ((\mathbf{x},\mathbf{y});\mathbf{v}_k)$
for some $(\mathbf{v}_k,\mathbf{w}_k) \in S$. As $S$ is compact, by passing to a subsequence if necessary,
we can assume that $(\mathbf{v}_k,\mathbf{w}_k) \rightarrow (\mathbf{v},\mathbf{w}) \in S$.
From Proposition \ref{prop:Lipschitz}, $DF_\uparrow(\mathbf{x},\mathbf{y})$ is Lipschitz; let its constant
be $l$. Hence, we have
$\mathbf{\overline{w}}_k \in \mathbf{w}_k + \mathbf{\widetilde{w}}_k +l \|\mathbf{v}_k -  \mathbf{v}\| \mathbf{B}$ where
$ \mathbf{\widetilde{w}}_k\in D F_\uparrow ((\mathbf{x},\mathbf{y});\mathbf{v})$. It follows that
$h_k  \mathbf{\widetilde{w}}_k \rightarrow \mathbf{\overline{w}} $. Hence, $\mathbf{\overline{w}} \in
D F_\uparrow ((\mathbf{x},\mathbf{y});\mathbf{v})^+ \cap -P$. But, from
Proposition~\ref{prop:prop2}, $D F_\uparrow ((\mathbf{x},\mathbf{y});\mathbf{v})$ is $P$-bounded,
hence $\mathbf{\overline{w}} = \mathbf{0}$.
\end{proof}  \newline

To be able to apply Proposition \ref{prop:prop4} in the setting of Proposition \ref{prop:super},
it suffices to show that the set $(\mathrm{FL})(\mathbf{x})$ is a nonempty compact subset of $\mathbf{R}^{n} \times
\mathbf{R}^{p}$.\\

\begin{lemma} \label{lem:lemma2}
The set $(\mathrm{FL})(\mathbf{x})$ is a nonempty compact subset of $\mathbf{R}^{n} \times
\mathbf{R}^{p}$ for all $\mathbf{x} \in \mathbf{R}^n$.
\end{lemma}
\begin{proof}
The set $(\mathrm{FL})(\mathbf{x})$ is closed by definition. The boundedness follows from the boundedness
of $\mathbf{f}$, see (\ref{eq:fbounded}), and the boundedness of $\mathbf{L}$, see (\ref{eq:Lbounded}).
\end{proof} \\

We can now proceed with the proof of Proposition \ref{prop:super}.\\
\begin{proof}[Proposition \ref{prop:super}] Assume that (\ref{eq:vfirst}) holds for
some $(t,\mathbf{x}) \in [0,T) \times \mathbf{R}^n$, $\mathbf{y} \in W(t,\mathbf{x}),$
and $(\mathbf{\overline{f}},\mathbf{\overline{L}}) \in (\mathrm{FL})(\mathbf{x})$. As
shown in Proposition \ref{prop:prop2}, the set $DW_\uparrow((t,\mathbf{x},\mathbf{y});(1,\mathbf{\overline{f}}))$
is $P$-bounded, hence externally stable from Proposition~\ref{prop:exter}. Therefore,
(\ref{eq:vfirst}) implies
$$  -\mathbf{\overline{L}} \in D_\uparrow W((t,\mathbf{x},\mathbf{y});(1,\overline{\mathbf{f}})) + P.  $$
Hence,
$$ -P \cap \bigg\{\mathbf{L} + D_\uparrow W((t,\mathbf{x},\mathbf{y});(1,\mathbf{f})), \ (\mathbf{f},\mathbf{L}) \in (\mathrm{FL})(\mathbf{x}) \bigg\} \neq \emptyset,$$
$$ -P \cap \mathrm{cl}\bigg( \bigg\{\mathbf{L} + D_\uparrow W((t,\mathbf{x},\mathbf{y});(1,\mathbf{f})), \ (\mathbf{f},\mathbf{L}) \in (\mathrm{FL})(\mathbf{x}) \bigg\} \bigg) \neq \emptyset,$$
 $$ \mathbf{0} \in \mathrm{cl}\bigg( \bigg\{\mathbf{L} + D_\uparrow W((t,\mathbf{x},\mathbf{y});(1,\mathbf{f})), \ (\mathbf{f},\mathbf{L}) \in (\mathrm{FL})(\mathbf{x}) \bigg\} \bigg)  + P.$$
From Proposition \ref{prop:prop4} together with Lemma \ref{lem:lemma2}, the following set is $P$-bounded: $$\mathrm{cl}\bigg( \bigg\{\mathbf{L} + D_\uparrow W((t,\mathbf{x},\mathbf{y});(1,\mathbf{f})), \ (\mathbf{f},\mathbf{L}) \in (\mathrm{FL})(\mathbf{x}) \bigg\} \bigg).$$
Hence, it is externally stable from Proposition~\ref{prop:exter}. Therefore,
 $$ \mathbf{0} \in \mathcal{E}\bigg(\mathrm{cl}\bigg( \bigg\{\mathbf{L} + D_\uparrow W((t,\mathbf{x},\mathbf{y});(1,\mathbf{f})), \ (\mathbf{f},\mathbf{L}) \in (\mathrm{FL})(\mathbf{x}) \bigg\} \bigg),P\bigg)  + P.$$
\end{proof}\\

It is possible to derive a converse to Proposition \ref{prop:super}. Again, we derive this result in a general setting.\\

\begin{proposition} \label{prop:prop5}
Under the assumptions of Proposition \emph{\ref{prop:prop4}}, assume that
$$ \mathbf{0} \in \mathcal{E}\bigg(\mathrm{cl}\bigg(\displaystyle \bigcup_{(\mathbf{v},\mathbf{w}) \in S} \mathbf{w} + D_\uparrow F((\mathbf{x},\mathbf{y});\mathbf{v})\bigg) \bigg\} \bigg),P\bigg)  + P.$$
Then there exists $(\mathbf{v},\mathbf{w}) \in S$ such that $$-\mathbf{w} \in D F_\uparrow((\mathbf{x},\mathbf{y});\mathbf{v}).$$
\end{proposition}
\begin{proof}
By hypothesis, some $\mathbf{d} \in P$ obeys $-\mathbf{d} \in \mathcal{E}\bigg(\mathrm{cl}\bigg(\displaystyle \bigcup_{(\mathbf{v},\mathbf{w}) \in S}
\mathbf{w} + D_\uparrow F((\mathbf{x},\mathbf{y});\mathbf{v})\bigg) \bigg\} \bigg),P\bigg).$
Then $-\mathbf{d} = \lim_{k \rightarrow +\infty}  \mathbf{w}_k +  \mathbf{\overline{w}}_k,$ where
$\mathbf{\overline{w}}_k \in D_\uparrow F((\mathbf{x},\mathbf{y});\mathbf{v}_k) \subset D F_\uparrow((\mathbf{x},\mathbf{y});\mathbf{v}_k)$ and $(\mathbf{v}_k,\mathbf{w}_k) \in S$.
As $S$ is compact, by passing to a subsequence if necessary, we can assume that $(\mathbf{v}_k,\mathbf{w}_k) \rightarrow (\mathbf{v},\mathbf{w}) \in S$.
Hence, $\mathbf{\overline{w}}_k$ converges to $-\mathbf{d} -\mathbf{w}$. From Proposition \ref{prop:Lipschitz}(2),
 $DF_\uparrow(\mathbf{x},\mathbf{y})$ is Lipschitz. Moreover, $DF_\uparrow(\mathbf{x},\mathbf{y})$ is a closed-valued
 map, hence it has closed graph. Therefore, $-\mathbf{d} -\mathbf{w} \in D F_\uparrow((\mathbf{x},\mathbf{y});\mathbf{v}),$
 or $ -\mathbf{w} \in D F_\uparrow((\mathbf{x},\mathbf{y});\mathbf{v}) + \mathbf{d},$
But, $D F_\uparrow((\mathbf{x},\mathbf{y});\mathbf{v}) = D F_\uparrow((\mathbf{x},\mathbf{y});\mathbf{v}) + P$.
Hence, $ -\mathbf{w} \in D F_\uparrow((\mathbf{x},\mathbf{y});\mathbf{v})$.
\end{proof}\\

\begin{corollary} \label{cor:super}
Let $W$ be an extremal element map, $(t,\mathbf{x}) \in [0,T) \times \mathbf{R}^n$,
and $\mathbf{y} \in W(t,\mathbf{x})$. Assume that $W$ is Lipschitz around $(t,\mathbf{x})$, that $\mathbf{y} \in \mathcal{PE}(W(t,\mathbf{x}),P)$,
and that \emph{(\ref{eq:vfirst_ext})} holds. Then  \emph{(\ref{eq:vfirst})}
holds for some $(\mathbf{\overline{f}},\mathbf{\overline{L}}) \in (\mathrm{FL})(\mathbf{x})$.
\end{corollary}

\subsection{A reformulation for (\ref{eq:vsecond})}

As above for (\ref{eq:vfirst}), we show that, when $W$ is Lipschitz around $(t,\mathbf{x}) \in (0,T] \times \mathbf{R}^n$ and
$\mathbf{y} \in W(t,\mathbf{x})$ is a properly minimal element with respect to $-P$ and $P$, i.e.,
$\mathbf{y} \in \mathcal{PE}(W(t,\mathbf{x}),-P) \cap \mathcal{PE}(W(t,\mathbf{x}),P)$, there exists a more compact formulation for (\ref{eq:vsecond}).\\

\begin{proposition} \label{prop:sub}
Let $W$ be a generalized contingent solution for \emph{(MOC)}, $(t,\mathbf{x}) \in (0,T] \times \mathbf{R}^n$,
and $\mathbf{y} \in W(t,\mathbf{x})$. Assume that:\\
\begin{enumerate}
  \item $W$ is Lipschitz around $(t,\mathbf{x})$,
  \item $\mathbf{y} \in \mathcal{PE}(W(t,\mathbf{x}),-P) \cap \mathcal{PE}(W(t,\mathbf{x}),P)$,
  \item the set-valued map $D_\uparrow W(t,\mathbf{x},\mathbf{y})$ is outer semicontinuous,
  \item the set $\mathcal{E}(\mathrm{cl}(S),-P)$ is closed, where
  \begin{equation} \label{eq:setS} S =  \bigg\{-\mathbf{L} + D_\uparrow W((t,\mathbf{x},\mathbf{y});(-1,-\mathbf{f})), \ (\mathbf{f},\mathbf{L}) \in (\mathrm{FL})(\mathbf{x}) \bigg\}.\end{equation}
\end{enumerate}
Then we have:
\begin{equation} \label{eq:vsecond_ext}
\mathbf{0} \in \mathcal{E}\bigg(\mathrm{cl}\bigg( \bigg\{-\mathbf{L} + D_\uparrow W((t,\mathbf{x},\mathbf{y});(-1,-\mathbf{f})), \ (\mathbf{f},\mathbf{L}) \in (\mathrm{FL})(\mathbf{x}) \bigg\} \bigg),-P\bigg)  + P,
\end{equation}
\end{proposition}

Knowing that $\mathbf{y} \in \mathcal{PE}(W(t,\mathbf{x}),P)$ and using Proposition \ref{prop:prop2}, it is possible to conclude that the sets $D W_\uparrow((t,\mathbf{x},\mathbf{y});(-1,-\mathbf{f}))$  are $P$-bounded.
Hence, (\ref{eq:vsecond}) becomes
$$  \mathbf{L} \in D_\uparrow W((t,\mathbf{x},\mathbf{y});(-1,-\mathbf{f})) +P,
 $$
 or
\begin{equation} \label{eq:vsecond_ext2} \forall (\mathbf{f},\mathbf{L})  \in  (\mathrm{FL})(\mathbf{x}), \ \mathbf{0} \in -\mathbf{L} +   D_\uparrow W((t,\mathbf{x},\mathbf{y});(-1,-\mathbf{f})) +P.
\end{equation}
Hence, proving Proposition \ref{prop:sub} amounts to proving that (\ref{eq:vsecond_ext2}) implies (\ref{eq:vsecond_ext}).
We provide a constructive proof below, but beforehand, we need the following two results, where
Proposition \ref{prop:Tanino} is derived in our usual general setting.\\

\begin{proposition}[Theorem 2.1, \emph{\emph{\cite{Tanino88_moo}}}] \label{prop:Tanino}
Let $(\mathbf{x},\mathbf{y}) \in Graph(F)$. Assume that $P$ has a compact base. Then,
$$ \forall \mathbf{v} \in X, \  D_\uparrow F((\mathbf{x},\mathbf{y});\mathbf{v}) \subset
DF((\mathbf{x},\mathbf{y});\mathbf{v}).$$
\end{proposition}

\begin{lemma} \label{label:prelim}
Under the assumptions of Proposition \emph{\ref{prop:sub}}, the set $\mathrm{cl}(S)$, where $S$ is defined by \emph{(\ref{eq:setS})} is $-P$-bounded. Moreover, if $\mathbf{w} \in \mathcal{E}(\mathrm{cl}(S),-P),$
then there exists $(\mathbf{\overline{f}},\mathbf{\overline{L}}) \in (\mathrm{FL})(\mathbf{x})$ such that $\mathbf{w} \in -\mathbf{\overline{L}} + D_\uparrow W((t,\mathbf{x},\mathbf{y});(-1,-\mathbf{\overline{f}}))$.
\end{lemma}
\begin{proof}
From Lemma \ref{lem:lemma1} (1)-(2) and Proposition \ref{prop:Tanino}, it is enough to show that the following set is
$-P$-bounded:
$$\bigg\{-\mathbf{L} + D W((t,\mathbf{x},\mathbf{y});(-1,-\mathbf{f})), \ (\mathbf{f},\mathbf{L}) \in (\mathrm{FL})(\mathbf{x}) \bigg\}.$$
From the last statement of Proposition \ref{prop:prop2} applied with $-P$, we get that the sets
$D W((t,\mathbf{x},\mathbf{y});(-1,-\mathbf{f}))$ are $-P$-bounded. From there, the proof follows the same lines as the
proof of Proposition \ref{prop:prop4}.\\ \\ Let  $\mathbf{w}_k \in S$ be a sequence converging to $\mathbf{w}$.
We have $\mathbf{w}_k \in -\mathbf{L}_k + D_\uparrow W((t,\mathbf{x},\mathbf{y});(-1,-\mathbf{f}_k))$ for some
$(\mathbf{f}_k,\mathbf{L}_k) \in (\mathrm{FL})(\mathbf{x})$. As $(\mathrm{FL})(\mathbf{x})$ is compact, we
can assume, by passing to a subsequence if necessary, that $(\mathbf{f}_k,\mathbf{L}_k)$ converges to
$(\mathbf{\overline{f}},\mathbf{\overline{L}}) \in (\mathrm{FL})(\mathbf{x})$. By assumption,
$D_\uparrow W(t,\mathbf{x},\mathbf{y})$ is outer semicontinuous. Hence $\mathbf{w}+ \mathbf{\overline{L}}
\in D_\uparrow W((t,\mathbf{x},\mathbf{y});(-1,-\mathbf{\overline{f}}))$, or
$\mathbf{w} \in -\mathbf{\overline{L}} + D_\uparrow W((t,\mathbf{x},\mathbf{y});(-1,-\mathbf{\overline{f}})).$
\end{proof} \\

We can now proceed with the proof of Proposition \ref{prop:sub}. Recall that $S$ is defined by (\ref{eq:setS}).\\

\begin{proof}[Proposition \ref{prop:sub}]
We  build sequences $\mathbf{f}_k$, $\mathbf{L}_k,$ $\mathbf{d}_k,$ and $\mathbf{w}_k,$ as follows. Take $(\mathbf{f}_k,\mathbf{L}_k) \in (\mathrm{FL})(\mathbf{x})$.
From (\ref{eq:vsecond_ext2}), there exists  $  \mathbf{d}_k \in (-\mathbf{L}_k +   D_\uparrow W((t,\mathbf{x},\mathbf{y});(-1,-\mathbf{f}_k))) \cap -P.$
Hence, $\mathbf{d}_k \in S.$ If $\mathbf{d}_k \in \mathcal{E}(\mathrm{cl}(S),-P),$ then
 the proof is complete. Otherwise, there exists $\mathbf{w}_k \in \mathrm{cl}(S)$
 such that $\mathbf{w}_k \in \mathbf{d}_k  + P \backslash \{\mathbf{0}\}. $ As $\mathrm{cl}(S)$ is
externally stable with respect to $-P$, without loss
 of generality, we can assume that $\mathbf{w}_k \in \mathcal{E}(\mathrm{cl}(S),-P)$.
 From Lemma \ref{label:prelim}, there exists $(\mathbf{f}_{k+1},\mathbf{L}_{k+1}) \in (\mathrm{FL})(\mathbf{x})$
 such that $\mathbf{w}_k \in -\mathbf{L}_{k+1} +   D_\uparrow W((t,\mathbf{x},\mathbf{y});(-1,-\mathbf{f}_{k+1})).$
If $\mathbf{w_k} \in -P$, the proof is complete. Otherwise, $\mathbf{w}_k \in -P^c.$
Repeat now the procedure  by choosing  $\mathbf{d}_{k+1} \in (-\mathbf{L}_{k+1} +   D_\uparrow W((t,\mathbf{x},\mathbf{y});(-1,-\mathbf{f}_{k+1}))) \cap -P$
using (\ref{eq:vsecond_ext2}). To summarize, the sequences  $\mathbf{f}_k$, $\mathbf{L}_k,$ $\mathbf{d}_k,$ and $\mathbf{w}_k,$
satisfy:\\ \\ Start with $(\mathbf{f}_0,\mathbf{L}_0) \in (\mathrm{FL})(\mathbf{x})$, $\mathbf{d}_0 \in (-\mathbf{L}_0 + D_\uparrow W((t,\mathbf{x},\mathbf{y});(-1,-\mathbf{f}_0))) \cap -P.$ Then, \\
 \begin{enumerate}
   \item $\mathbf{w}_k \in \mathbf{d}_k  + P \backslash \{\mathbf{0}\}$,
   \item $\mathbf{w}_k \in \mathcal{E}(\mathrm{cl}(S),-P)$,
   \item $\mathbf{w}_k \in -\mathbf{L}_{k+1} +   D_\uparrow W((t,\mathbf{x},\mathbf{y});(-1,-\mathbf{f}_{k+1}))$ for some $(\mathbf{f}_{k+1},\mathbf{L}_{k+1}) \in (\mathrm{FL})(\mathbf{x})$,
   \item $\mathbf{w}_k \in -P^c,$
   \item $\mathbf{d}_{k+1} \in -\mathbf{L}_{k+1} +   D_\uparrow W((t,\mathbf{x},\mathbf{y});(-1,-\mathbf{f}_{k+1}))$,
   \item $\mathbf{d}_{k+1} \in -P$.\\
 \end{enumerate}
As $(\mathrm{FL})(\mathbf{x})$ is compact, by passing to a subsequence if necessary, we can assume that $(\mathbf{f}_{k},\mathbf{L}_{k})$ converges to
$(\mathbf{\overline{f}},\mathbf{\overline{L}}) \in (\mathrm{FL})(\mathbf{x}).$ Consider now the recession cone of the set $\{(\mathbf{d}_k,\mathbf{w}_k), \ k \in \mathbf{N}\}$. Take $(\mathbf{\overline{d}}, \mathbf{\overline{w}}) \in
\{(\mathbf{d}_k,\mathbf{w}_k), \ k \in \mathbf{N}\}^+ $, i.e.,
let $h_k \rightarrow 0^+$ be such that $h_k \mathbf{d_k} \rightarrow \mathbf{\overline{d}}$
and $h_k \mathbf{w}_k \rightarrow \mathbf{\overline{w}}$. From (1) and (6) above, we get
\begin{equation} \label{eq:atthelimit1}
 \mathbf{\overline{d}}- \mathbf{\overline{w}} \in  - P,
\end{equation}
\begin{equation} \label{eq:atthelimit2}
\mathbf{\overline{d}} \in -P,
\end{equation}
We show now that necessarily $( \mathbf{\overline{w}}, \mathbf{\overline{z}}) = (\mathbf{0},\mathbf{0})$.
From Proposition \ref{prop:Tanino},  we have $D_\uparrow W((t,\mathbf{x},\mathbf{y});(-1,-\mathbf{f}_{k+1}))
 \subset DW((t,\mathbf{x},\mathbf{y});(-1,-\mathbf{f}_{k+1}))$.
Hence, $\mathbf{w}_k + \mathbf{L}_{k+1}  \in DW((t,\mathbf{x},\mathbf{y});(-1,-\mathbf{f}_{k+1})).$
Moreover, from Proposition \ref{prop:Lipschitz}, the set-valued map $DW(\mathbf{x},\mathbf{y})$ is Lipschitz. Hence,
along the lines of the proof of Proposition \ref{prop:prop4}, it can be shown that
$\mathbf{\overline{w}} \in DW((t,\mathbf{x},\mathbf{y});(-1,-\mathbf{\overline{f}}))^+$.
A similar argument also gives $\mathbf{\overline{d}} \in DW((t,\mathbf{x},\mathbf{y});(-1,-\mathbf{\overline{f}}))^+$.
Hence, from (\ref{eq:atthelimit2}), $\mathbf{\overline{d}} \in DW((t,\mathbf{x},\mathbf{y});(-1,-\mathbf{\overline{f}}))^+ \cap  -P$, but we know that the set
$DW((t,\mathbf{x},\mathbf{y});(-1,-\mathbf{\overline{f}}))$ is $P$-bounded, hence $\mathbf{\overline{d}} = \mathbf{0}$.
Therefore,  from (\ref{eq:atthelimit1}), we get that  $\mathbf{\overline{w}} \in  P$.
Hence, $\mathbf{\overline{w}} \in DW((t,\mathbf{x},\mathbf{y});(-1,-\mathbf{\overline{f}}))^+\cap P$.
But, we know that the set $DW((t,\mathbf{x},\mathbf{y});(-1,-\mathbf{\overline{f}}))$ is $-P$-bounded,
hence $\mathbf{\overline{w}} = \mathbf{0}$. Finally, $( \mathbf{\overline{d}}, \mathbf{\overline{w}}) = (\mathbf{0},\mathbf{0})$, and therefore
the set $\{(\mathbf{d}_k,\mathbf{w}_k), \ k \in \mathbf{N}\}$ is bounded (Lemma 3.2.1 \cite[p.~52]{NonlinearMOO1_book}). \\ \\
By passing to a subsequence if necessary, we can therefore assume that $\mathbf{d}_k$ converges to $\mathbf{\overline{d}}$ and that
$\mathbf{w}_k$ converges to $\mathbf{\overline{w}}$. By assumption, $D_\uparrow W(t,\mathbf{x},\mathbf{y})$
is outer semicontinuous, hence $\mathbf{\overline{d}}+\mathbf{\overline{L}},\mathbf{\overline{w}}+\mathbf{\overline{L}}
\in D_\uparrow W((t,\mathbf{x},\mathbf{y});(-1,-\mathbf{\overline{f}}))$. From (1)  above, we get $\mathbf{\overline{d}}-
\mathbf{\overline{w}} \in  - P$, or $\mathbf{\overline{d}}+ \mathbf{\overline{L}}- (\mathbf{\overline{w}}+\mathbf{\overline{L}}) \in  - P$.
$\mathbf{\overline{d}}+ \mathbf{\overline{L}}$ and $\mathbf{\overline{w}}+\mathbf{\overline{L}}$ are both minimal elements, hence we must have
$\mathbf{\overline{d}}+ \mathbf{\overline{L}} = \mathbf{\overline{w}}+ \mathbf{\overline{L}}$,
or $\mathbf{\overline{d}} = \mathbf{\overline{w}}$. From (6)  above, we get $\mathbf{\overline{d}} \in -P$,
hence $\mathbf{\overline{w}} \in -P$. Finally, recall that $\mathbf{\overline{w}}$ is the limit of the sequence
 $\mathbf{\overline{w}_k} \in \mathcal{E}(\mathrm{cl}(S),-P),$ which is a closed set by assumption.
Hence, $\mathbf{\overline{w}} \in \mathcal{E}(\mathrm{cl}(S),-P) \cap -P$, which completes the proof.
\end{proof}

\subsection{The case $p=1$} \label{ss:casepegal1}

In this section, we show that, when $p=1$ and $P =\mathbf{R}^+$, the notion of lower Dini solution
for single objective optimal control problems is retrieved from Definition \ref{def:solution}. In this case, $W$ is a set-valued map
from $I \times \mathbf{R}^n$ to $\mathbf{R}$. As, by definition, $W$ is an extremal element map, it follows
from $(\ref{eq:property1})$ and $(\ref{eq:property2})$ that the values of $W$ reduce to singletons. Hence, we can write $W(t,\mathbf{x}) = \{w(t,\mathbf{x})\},$ where $w$ is a function from $I \times \mathbf{R}^n$ to $\mathbf{R}$.\\

To go further, we need the following proposition which relates the contingent epiderivative  and the lower generalized Dini derivative for extended real-valued functions.\\

\begin{proposition}[\emph{\emph{\cite[p.~225]{SVA_book}}}] \label{prop:dinieq}
Let $(t,\mathbf{x}) \in  I \times \mathbf{R}^n$ and $(\tau,\mathbf{v}) \in  \mathbf{R}^{n+1}.$
Then, by definition of the contingent epiderivative $D_\uparrow w(t,\mathbf{x}): \mathbf{R}^{n+1} \rightarrow
\mathbf{R} \cup \{-\infty,
+\infty\}$:
\begin{equation} \label{eq:defcontinderi}
\forall u \in D w_\uparrow((t,\mathbf{x});(\tau,\mathbf{v})), \ D_\uparrow w((t,\mathbf{x});(\tau,\mathbf{v})) \leq u.
\end{equation}
Moreover,
$$D_\uparrow w((t,\mathbf{x});(\tau,\mathbf{v})) =  \liminf_{h\rightarrow 0^+,(s,\mathbf{p})\rightarrow(\tau,\mathbf{v})} \frac {w(t+hs,\mathbf{x}+h\mathbf{p})-w(t,\mathbf{x})}{h}.$$ Hence,
\begin{equation} \label{eq:eqcondini}
D_\uparrow w((t,\mathbf{x});(\tau,\mathbf{v})) =  \partial^-w((t,\mathbf{x});(\tau,\mathbf{v})),
\end{equation}
where $\partial^-w((t,\mathbf{x});(\tau,\mathbf{v}))$ denotes
the lower generalized Dini derivative of $w$ at $(t,\mathbf{x})$ in the direction
$(\tau,\mathbf{v})$. \\ \\
\end{proposition}

As $DW_\uparrow((t,\mathbf{x},w(t,\mathbf{x}));(\tau,\mathbf{v})) = Dw_\uparrow((t,\mathbf{x});(\tau,\mathbf{v}))$,
(\ref{eq:vfirst}) therefore  writes
$$
 -\overline{L} \in Dw_\uparrow((t,\mathbf{x});(1,\mathbf{\overline{f}})).
$$
Hence, from (\ref{eq:defcontinderi})-(\ref{eq:eqcondini}),
$$
\partial^-w((t,\mathbf{x});(1,\mathbf{\overline{f}})) \leq -\overline{L},
$$
or
\begin{equation} \label{eq:dini1}
\inf_{(\mathbf{f},L) \in (\mathrm{FL})(\mathbf{x})}  L + \partial^-w((t,\mathbf{x});(1,\mathbf{f})) \leq 0.
\end{equation}
Accordingly, (\ref{eq:vsecond}) writes
$$
\forall (\mathbf{f},L) \in (\mathrm{FL})(\mathbf{x}), \ L \in Dw_\uparrow((t,\mathbf{x});(-1,-\mathbf{f})).
$$
Hence, from (\ref{eq:defcontinderi})-(\ref{eq:eqcondini}),
$$
\forall (\mathbf{f},L) \in (\mathrm{FL})(\mathbf{x}), \ \partial^-w((t,\mathbf{x});(-1,-\mathbf{f})) \leq L,
$$
or
\begin{equation} \label{eq:dini2}
\sup_{(\mathbf{f},L) \in (\mathrm{FL})(\mathbf{x})}  -L + \partial^-w((t,\mathbf{x});(-1,-\mathbf{f})) \leq 0.
\end{equation}
Finally,
\begin{equation} \label{eq:dinitercond}
w(T,\mathbf{x}) = 0.
\end{equation}
Equations (\ref{eq:dini1})-(\ref{eq:dinitercond}) precisely correspond to the notion of lower Dini solution for single objective optimal control problems.\\

Equations (\ref{eq:dini1}) and (\ref{eq:dini2}) can also be directly obtained from Propositions \ref{prop:super} and
\ref{prop:sub}. First, observe that $w(t,\mathbf{x})$ is a properly minimal element of the set $W(t,\mathbf{x})$ with respect to -$P$ and $P$, i.e., $w(t,\mathbf{x}) \in \mathcal{PE}(W(t,\mathbf{x}),-P) \cap \mathcal{PE}(W(t,\mathbf{x}),P)$. Assume now that $W$ is Lipschitz  around $(t,\mathbf{x})$, then it follows that $w$ is also Lipschitz. Hence, from Propositions~\ref{prop:prop2} and \ref{prop:exter}, the sets
$D_\uparrow W((t,\mathbf{x},w(t,\mathbf{x}));(\tau,\mathbf{v}))$
are nonempty and  reduce to singletons. More precisely,
we have
$$
 D_\uparrow W((t,\mathbf{x},  w(t,\mathbf{x}) );(\tau,\mathbf{v})) = \{ D_\uparrow w((t,\mathbf{x} );(\tau,\mathbf{v})) \}.
 $$
Hence, (\ref{eq:vfirst_ext}) from Proposition \ref{prop:super} yields
$$
0 \in \mathcal{E}\bigg(\mathrm{cl}\bigg( \bigg\{L +  \partial^-w((t,\mathbf{x});(1,\mathbf{f})), \ (\mathbf{f},L) \in (\mathrm{FL})(\mathbf{x}) \bigg\} \bigg),P\bigg)  + P,
$$
which is precisely  (\ref{eq:dini1}).  For Proposition \ref{prop:sub}, note that Assumptions 3 and 4 are automatically satisfied. Indeed, the set-valued map $D_\uparrow W(t,\mathbf{x},w(t,\mathbf{x}))$  is Lipschitz and takes closed values, i.e., singletons, hence it is outer semicontinuous. Moreover, the set $\mathcal{E}(\mathrm{cl}(S),-P)$ in Assumption 4 reduces to a singleton, hence it is closed. Finally, (\ref{eq:vsecond_ext}) yields
$$
0 \in \mathcal{E}\bigg(\mathrm{cl}\bigg( \bigg\{-L +  \partial^-w((t,\mathbf{x});(-1,-\mathbf{f})), \ (\mathbf{f},L) \in (\mathrm{FL})(\mathbf{x}) \bigg\} \bigg),-P\bigg)  + P,
$$
which is precisely  (\ref{eq:dini2}).

\section{The set-valued return function $V$ is a  generalized contingent solution for (MOC)} \label{s:setvalued}

In this section, we first show that the set-valued return function is a generalized contingent solution
for (MOC). We then prove that the set-valued map $V_\uparrow$ is outer semicontinuous.\\

\begin{proposition} \label{prop:setvaluedmap}
The set-valued return function $V$  is a generalized contingent solution
for (MOC).
\end{proposition}
\begin{proof}
First, we need to show that $V$ is an extremal element map.  The fact that (\ref{eq:property1})  and (\ref{eq:property2}) hold follows directly from the definition of $V$. It remains to show that $V$ satisfies (\ref{eq:vfirst})-(\ref{eq:tercon1}). The proofs for (\ref{eq:vfirst}) and (\ref{eq:vsecond}) are respectively postponed to \S \ref{ss:SVAsuper} and \S \ref{ss:SVAsub}. For (\ref{eq:tercon1}), this again follows directly from the definition of $V$.
\end{proof}\\

\begin{proposition} \label{prop:outersemi}
The set-valued map $V_\uparrow$ is outer semicontinuous on $I \times \mathbf{R}^n$.
\end{proposition}
\begin{proof}
Let $(t,\mathbf{x}) \in I \times \mathbf{R}^n$. From Corollary \ref{cor:extstabcompactset}, the set $\mathrm{cl}(Y(t,\mathbf{x}))$ is externally stable. Hence, $$V(t,\mathbf{x}) + P =  \mathrm{cl}(Y(t,\mathbf{x})) +P.$$ Therefore, showing that the set-valued map $V_\uparrow$ is outer semicontinuous at $(t,\mathbf{x})$ amounts to showing that the set-valued map $\mathrm{cl}(Y(\cdot,\cdot)) +P$ is outer semicontinuous at $(t,\mathbf{x})$. But, this follows from
the fact that the set $\mathrm{cl}(Y(t,\mathbf{x})) +P$ is closed and the set-valued map $Y$ is Lipschitz
(Corollary \ref{prop:objspesti}).
\end{proof}

\subsection{$V$  satisfies (\ref{eq:vfirst})} \label{ss:SVAsuper}
For the developments below, we need the two   estimates contained in the following lemma.\\

\begin{lemma} \label{lem:estimate}
For $(t,\mathbf{x}) \in [0,T) \times \mathbf{R}^n$ and $\mathbf{u}(\cdot) \in \mathcal{U}$, one has
\begin{equation} \label{eq:firstest}
\frac{1}{\tau} \int_t^{t+\tau} \mathbf{f}(\mathbf{x}(s;t,\mathbf{x},\mathbf{u}(\cdot)),\mathbf{u}(s)) \ \mathrm{d}s  = \frac{1}{\tau}  \int_t^{t+\tau} \mathbf{f}(\mathbf{x},\mathbf{u}(s)) \ \mathrm{d}s + o(\mathbf{1}),
\end{equation}
and
\begin{equation} \label{eq:secondest}
\frac{1}{\tau} \int_t^{t+\tau} \mathbf{L}(\mathbf{x}(s;t,\mathbf{x},\mathbf{u}(\cdot)),\mathbf{u}(s)) \ \mathrm{d}s  = \frac{1}{\tau}  \int_t^{t+\tau} \mathbf{L}(\mathbf{x},\mathbf{u}(s)) \ \mathrm{d}s + o(\mathbf{1}),
\end{equation}
as $\tau \rightarrow 0^+$ independently of $\mathbf{u}(\cdot)$.
\end{lemma}
\begin{proof}
The first estimate (\ref{eq:firstest}) follows from the Lipschitz assumption, see (\ref{eq:lipf}), and the boundedness assumption, see (\ref{eq:fbounded}), on $\mathbf{f}$,
while the second estimate (\ref{eq:secondest}) follows from the Lipschitz assumption on $\mathbf{L}$, see (\ref{eq:LLipschitz}), and the boundedness assumption
 on $\mathbf{f}$, see (\ref{eq:fbounded}).
\end{proof}\\

Let $(t,\mathbf{x}) \in [0,T) \times \mathbf{R}^n$ and $\mathbf{y} \in V(t,\mathbf{x})$.
We have  $\mathbf{y} \in V(t,\mathbf{x}) \subset \mathrm{cl}(Y(t,\mathbf{x}))$. Hence,
$\forall \epsilon_k \rightarrow 0^+,$ there exists $\mathbf{\widetilde{y}}_k\in Y(t,\mathbf{x})$
such that $ \|\mathbf{y} - \mathbf{\widetilde{y}}_k \| < \epsilon_k$. Choose a sequence $\tau_k \rightarrow 0^+$
such that $\epsilon_k /\tau_k \rightarrow 0^+$ (e.g. $\tau_k = \sqrt{\epsilon_k}$).
From Lemma \ref{lem:dp2}, there exists $\mathbf{u}_k(\cdot) \in \mathcal{U}$ such that
\begin{displaymath}
\mathbf{\widetilde{y}}_k \in \int_t^{t+\tau_k} \mathbf{L}(\mathbf{x}(s;t,\mathbf{x},\mathbf{u}_k(\cdot)),\mathbf{u}_k(s)) \ \mathrm{d}s +
V(t+\tau_k,\mathbf{x}(t+\tau_k;t,\mathbf{x},\mathbf{u}_k(\cdot))) + P,\end{displaymath}
or
\begin{equation} \label{eq:start2_1}
\mathbf{y} \in \int_t^{t+\tau_k} \mathbf{L}(\mathbf{x}(s;t,\mathbf{x},\mathbf{u}_k(\cdot)),\mathbf{u}_k(s)) \ \mathrm{d}s +
V(t+\tau_k,\mathbf{x}(t+\tau_k;t,\mathbf{x},\mathbf{u}_k(\cdot))) + P + \epsilon_k \mathbf{B}.\end{equation}
From (\ref{eq:firstest}) and (\ref{eq:secondest}), we have:
$$ \mathbf{x}(t+\tau_k;t,\mathbf{x},\mathbf{u}_k(\cdot))  = \mathbf{x} + \int_t^{t+\tau_k} \mathbf{f}(\mathbf{x},\mathbf{u}_k(s)) \ \mathrm{d}s + \tau_k o(\mathbf{1}),$$
and
$$ \int_t^{t+\tau_k} \mathbf{L}(\mathbf{x}(s;t,\mathbf{x},\mathbf{u}_k(\cdot)),\mathbf{u}_k(s)) \ \mathrm{d}s  =
\int_t^{t+\tau_k} \mathbf{L}(\mathbf{x},\mathbf{u}_k(s)) \ \mathrm{d}s + \tau_k o(\mathbf{1}),$$
independently from $\mathbf{u}_k(\cdot)$. Moreover,
$$\bigg(\frac{1}{\tau_k}\displaystyle \int_t^{t+\tau_k} \mathbf{L}(\mathbf{x},\mathbf{u}_k(s)) \ \mathrm{d}s, \frac{1}{\tau_k}\int_t^{t+\tau_k} \mathbf{f}(\mathbf{x},\mathbf{u}_k(s)) \ \mathrm{d}s\bigg)  \in (\mathrm{FL})(\mathbf{x}).$$
From Lemma \ref{lem:lemma2}, the set $(\mathrm{FL})(\mathbf{x})$ is compact. Hence, by passing to
a subsequence if necessary, we can assume  that there exists $(\mathbf{\overline{f}},\mathbf{\overline{L}}) \in (\mathrm{FL})(\mathbf{x})$ such that
$$\bigg(\frac{1}{\tau_k}\displaystyle \int_t^{t+\tau_k} \mathbf{L}(\mathbf{x},\mathbf{u}_k(s)) \ \mathrm{d}s, \frac{1}{\tau_k}\int_t^{t+\tau_k} \mathbf{f}(\mathbf{x},\mathbf{u}_k(s)) \ \mathrm{d}s\bigg) \rightarrow (\mathbf{\overline{f}},\mathbf{\overline{L}})\ \mathrm{as} \ k \rightarrow \infty.$$
Hence, we can write
$$ \mathbf{x}(t+\tau_k;t,\mathbf{x},\mathbf{u}_k(\cdot)) = \mathbf{x} + \tau_k (\mathbf{\overline{f}} + o(\mathbf{1})),$$
and
$$\int_t^{t+\tau_k} \mathbf{L}(\mathbf{x}(s;t,\mathbf{x},\mathbf{u}_k(\cdot)),\mathbf{u}_k(s)) \ \mathrm{d}s = \tau_k(\mathbf{\overline{L}}  + o(\mathbf{1})).$$
 Substituting these two equalities into $(\ref{eq:start2_1})$ yields
 $$
 \mathbf{y}  + \tau_k(-\mathbf{\overline{L}}+ o(\mathbf{1})) \in  V(t+ \tau_k,\mathbf{x} + \tau_k (\mathbf{\overline{f}} + o(\mathbf{1}))) + P+\epsilon_k \mathbf{B},
 $$or using the fact that $\epsilon_k = o(\tau_k)$,
 $$
 \mathbf{y}  + \tau_k(-\mathbf{\overline{L}}+ o(\mathbf{1})) \in  V(t+ \tau_k,\mathbf{x} + \tau_k (\mathbf{\overline{f}} + o(\mathbf{1}))) + P.
 $$
But this precisely corresponds to the definition of the contingent
derivative of $V_\uparrow$ at $(t,\mathbf{x},\mathbf{y})$  in the direction $(1,\mathbf{\overline{f}}).$ Hence,
$$
 -\mathbf{\overline{L}} \in DV_\uparrow((t,\mathbf{x},\mathbf{y});(1,\overline{\mathbf{f}})),
$$
which is (\ref{eq:vfirst}).

\subsection{$V$  satisfies (\ref{eq:vsecond})} \label{ss:SVAsub}
For the developments below, we need  the following technical lemma.\\

\begin{lemma}[\emph{\emph{\cite[p.~129]{OptConTh8_book}}}] \label{lem:flx}
Let $(t,\mathbf{x}) \in (0,T] \times \mathbf{R}^n,$  we have
\begin{multline} \nonumber
(\mathrm{FL})(\mathbf{x}) = \bigg\{\mathbf{g} = (\mathbf{f},\mathbf{L}) \in \mathbf{R}^{n+p}: \\ \mathbf{g} = \lim_{k\rightarrow\infty} \frac{1}{\tau_k} \int_{t-\tau_k}^t (\mathbf{f}(\mathbf{x},\mathbf{u}_k(s)),\mathbf{L}(\mathbf{x},\mathbf{u}_k(s))) \ \mathrm{d}s \ \mathrm{for} \ \mathrm{some} \ \tau_k\rightarrow 0^+ \ \mathrm{and} \ \mathbf{u}_k(\cdot) \in \mathcal{U} \bigg\}.
 \end{multline}
\end{lemma}
Let $(t,\mathbf{x}) \in (0,T] \times \mathbf{R}^n$, $\mathbf{u}(\cdot) \in \mathcal{U}$, and $\mathbf{y} \in V(t,\mathbf{x})$.
Let $\delta \rightarrow \mathbf{z}(\delta;0,\mathbf{x},\mathbf{u}(\cdot))$ be the solution to
$$
\left \{ \begin{array}{lcl} \mathbf{\dot{z}}(\delta)& = & -\mathbf{f}(\mathbf{z}(\delta),\mathbf{u}(t - \delta)), \ 0 \leq \delta  \leq \tau, \\
\mathbf{z}(0) & = & \mathbf{x}.\end{array}\right.
$$
The function $\delta \rightarrow \mathbf{z}(\delta;0,\mathbf{x},\mathbf{u}(\cdot))$ and the trajectory
$s \rightarrow \mathbf{x}(s;t-\tau,\mathbf{z}(\tau;0,\mathbf{x},\mathbf{u}(\cdot)),\mathbf{u}(\cdot))$ satisfy
 $$\mathbf{x}(t-\delta;t-\tau,\mathbf{z}(\tau;0,\mathbf{x},\mathbf{u}(\cdot)),\mathbf{u}(\cdot))
 = \mathbf{z}(\delta;0,\mathbf{x},\mathbf{u}(\cdot)),
 \ 0 \leq \delta  \leq \tau,$$ and therefore,
 $\mathbf{x}(t;t-\tau,\mathbf{z}(\tau;0,\mathbf{x},\mathbf{u}(\cdot)),\mathbf{u}(\cdot))
 = \mathbf{z}(0;0,\mathbf{x},\mathbf{u}(\cdot))= \mathbf{x}$. Hence, applying Lemma \ref{lem:dp1}
to the trajectory $s \rightarrow \mathbf{x}(s;t-\tau,\mathbf{z}(\tau;0,\mathbf{x},\mathbf{u}(\cdot)),\mathbf{u}(\cdot))$
on the interval $[t-\tau,t]$ gives
\begin{equation} \nonumber
\int_{t-\tau}^t \mathbf{L}(\mathbf{x}(s;t-\tau,\mathbf{z}(\tau;0,\mathbf{x},\mathbf{u}(\cdot)),\mathbf{u}(\cdot)),\mathbf{u}(s)) \ \mathrm{d}s + V(t,\mathbf{x}) \subset \mathrm{cl}(Y(t-\tau,\mathbf{z}(\tau;0,\mathbf{x},\mathbf{u}(\cdot)))),
\end{equation}
or,  as $\mathrm{cl}(Y(t,\mathbf{x}))$ is externally stable from Corollary \ref{cor:extstabcompactset},
\begin{equation} \label{eq:start2}
\int_{t-\tau}^t \mathbf{L}(\mathbf{x}(s;t-\tau,\mathbf{z}(\tau;0,\mathbf{x},\mathbf{u}(\cdot)),\mathbf{u}(\cdot)),\mathbf{u}(s)) \ \mathrm{d}s + V(t,\mathbf{x}) \subset V(t-\tau,\mathbf{z}(\tau;0,\mathbf{x},\mathbf{u}(\cdot))) +P.
\end{equation}
Now, let $(\mathbf{f},\mathbf{L}) \in (\mathrm{FL})(\mathbf{x})$.  From Lemma~\ref{lem:flx},  there exist
$\tau_k \rightarrow 0^+$ and $\mathbf{u}_k(\cdot) \in \mathcal{U}$ such that
$$ \mathbf{f} = \lim_{k\rightarrow\infty} \frac{1}{\tau_k} \int_{t-\tau_k}^{t} \mathbf{f}(\mathbf{x},\mathbf{u}_k(s)) \ \mathrm{d}s, \ \mathrm{and} \  \mathbf{L} = \lim_{k\rightarrow\infty} \frac{1}{\tau_k} \int_{t-\tau_k}^{t} \mathbf{L}(\mathbf{x},\mathbf{u}_k(s)) \ \mathrm{d}s.$$
We have:
$$\mathbf{z}(\tau_k;0,\mathbf{x},\mathbf{u}_k(\cdot))  = \mathbf{x} - \int_0^{\tau_k} \mathbf{f}(\mathbf{z}(\delta;0,\mathbf{x},\mathbf{u}_k(\cdot)),\mathbf{u}_k(t - \delta)) \ \mathrm{d}\delta.$$
As in \S \ref{ss:SVAsuper}, it can be shown that
$$ \mathbf{z}(\tau;0,\mathbf{x},\mathbf{u}(\cdot))  = \mathbf{x} - \tau_k(\mathbf{f}+o(\mathbf{1})),
$$
and
$$
\int_{t-\tau_k}^t \mathbf{L}(\mathbf{x}(s;t-\tau_k,\mathbf{z}(\tau_k;0,\mathbf{x},\mathbf{u}_k(\cdot)),\mathbf{u}_k(\cdot)),\mathbf{u}_k(s)) \ \mathrm{d}s= \tau_k(\mathbf{L}  + o(\mathbf{1})),$$
as $\tau_k \rightarrow 0^+$ independently of $\mathbf{u}_k(\cdot)$.
Substituting these two equalities into (\ref{eq:start2}) yields
$$
\tau_k(\mathbf{L}  + o(\mathbf{1})) + V(t,\mathbf{x})  \in V(t-\tau_k,\mathbf{x}-\tau_k(\mathbf{f}  + o(\mathbf{1})))   + P,
$$
or
$$
\mathbf{y} + \tau_k(\mathbf{L}  + o(\mathbf{1}))  \in V(t-\tau_k,\mathbf{x}-\tau_k(\mathbf{f}  + o(\mathbf{1})) )   + P.
$$
But this precisely corresponds to the definition of the contingent
derivative of $V_\uparrow$ at $(t,\mathbf{x},\mathbf{y})$  in the direction $(-1,-\mathbf{f}).$ Hence,
$$
\forall (\mathbf{f},\mathbf{L}) \in (\mathrm{FL})(\mathbf{x}), \  \mathbf{L} \in DV_\uparrow((t,\mathbf{x},\mathbf{y});(-1,-\mathbf{f})),
$$
which is (\ref{eq:vsecond}).

\section{Generalized proximal solution for (MOC)} \label{s:proximal}

The notion of proximal solution to the Hamilton-Jacobi Equation for
finite-horizon single objective optimal control problems was introduced
in \cite[p.~454]{OptConTh5_book}.
In this section, using the concept of coderivative for set-valued maps,
we extend this notion to (MOC). We call this extended notion generalized proximal
solution for (MOC).

\subsection{Definition}

\begin{definition} \label{def:genproximalsolution}
An extremal element map from $I \times \mathbf{R}^n$ to $\mathbf{R}^p$ is said to be a generalized proximal solution
for  \emph{(MOC)} if:\\
\begin{itemize}
  \item  For all $(t,\mathbf{x}) \in (0,T) \times \mathbf{R}^n$, for all $\mathbf{y} \in W(t,\mathbf{x}),$
and for all $\mathbf{w^*} \in \mathbf{R}^p$ such that $D^*W_\uparrow(t,\mathbf{x},\mathbf{y})(\mathbf{w^*}) \neq \emptyset$,
 \begin{equation} \label{eq:proximal_equation}
\forall  (\xi^*,\mathbf{v^*}) \in D^*W_\uparrow(t,\mathbf{x},\mathbf{y})(\mathbf{w^*}), \ \xi^*+
\inf_{(\mathbf{f},\mathbf{L}) \in (\mathrm{FL})(\mathbf{x})} \langle\mathbf{
v^*},\mathbf{f}  \rangle + \langle \mathbf{w^*}, \mathbf{L} \rangle = 0.
\end{equation}
  \item For all $\mathbf{x} \in  \mathbf{R}^n$, \begin{equation} \label{eq:boundary_condition}
  W_\uparrow(0,\mathbf{x}) \subset \limsup_{t \rightarrow  0^+, \ \mathbf{x'}\rightarrow \mathbf{x}} W_\uparrow(t,\mathbf{x'}) \ \mathrm{and} \
  W_\uparrow(T,\mathbf{x}) \subset \limsup_{t \rightarrow  T^-,\ \mathbf{x'}\rightarrow \mathbf{x}} W_\uparrow(t,\mathbf{x'}).\end{equation}
  \item For all $\mathbf{x} \in  \mathbf{R}^n$,
\begin{equation} \label{eq:tercon2}
W(T,\mathbf{x}) = \{\mathbf{0}\}.
\end{equation}
\end{itemize}
\end{definition}

\subsection{The case $p=1$}

In \S \ref{ss:casepegal1}, we have shown that  $W(t,\mathbf{x}) = \{w(t,\mathbf{x})\},$
where $w$ is a function from $I \times \mathbf{R}^n$ to $\mathbf{R}$. Hence,
recalling that $P = \mathbf{R}^+$, we have $\mathrm{gph}(W_\uparrow) = \mathrm{epi}(w)$. Therefore,
 $N_{\mathrm{gph}(W_\uparrow)}(t,\mathbf{x}, w(t,\mathbf{x})) = N_{\mathrm{epi}(w)}(t,\mathbf{x},w(t,\mathbf{x})),$
or
\begin{equation} \label{eq:equality}
\partial w(t,\mathbf{x}) = D^*W_\uparrow(t,\mathbf{x},w(t,\mathbf{x}))(1),
\end{equation}
where $\partial w(t,\mathbf{x})$ is the proximal subdifferential of $w$ at $(t,\mathbf{x})$ \cite[P.~135]{OptConTh5_book}. Note that, again, we use the notation $\partial$  for the proximal subdifferential instead of the  traditional notation $\partial^P$ to avoid any possible confusion with the ordering cone $P$.
Substituting (\ref{eq:equality}) in  (\ref{eq:proximal_equation}) yields
\begin{equation} \nonumber
\forall  (\xi^*,\mathbf{v^*}) \in
\partial w(t,\mathbf{x}), \ \xi^*+ \inf_{(\mathbf{f},L) \in (\mathrm{FL})(\mathbf{x})} \langle\mathbf{v^*},\mathbf{f}  \rangle +  L  = 0,
\end{equation}
To retrieve the notion of proximal solution for
finite-horizon single objective optimal control problems, it remains to show
that~(\ref{eq:boundary_condition}) yields:
$$ w(0,\mathbf{x}) \geq \liminf_{t \rightarrow  0^+, \ \mathbf{x'}\rightarrow \mathbf{x}} w(t,\mathbf{x'}) \ \mathrm{and} \  w(T,\mathbf{x}) \geq \liminf_{t \rightarrow  T^-, \ \mathbf{x'}\rightarrow \mathbf{x}} w(t,\mathbf{x'}).$$
We only prove below the first inequality. From (\ref{eq:boundary_condition}), there exist $t_k \rightarrow 0^+, \  \mathbf{x}_k \rightarrow \mathbf{x}, \ y_k \rightarrow w(0,\mathbf{x})$ such that $y_k \in W_\uparrow(t_k,\mathbf{x}_k) = \{w(t_k,\mathbf{x}_k)\} + \mathbf{R}^+$. Hence, $y_k \geq w(t_k,\mathbf{x}_k)$,
from which follows that $w(0,\mathbf{x}) \geq \liminf_{t \rightarrow  0^+, \ \mathbf{x'}\rightarrow \mathbf{x}} w(t,\mathbf{x'})$.

\subsection{From generalized contingent solution to generalized proximal solution}

In this section, we show that if a set-valued map is a generalized contingent solution for (MOC), then it is also
a generalized proximal solution for (MOC).\\

\begin{proposition} \label{prop:equivalence}
Let $W$ be a generalized contingent solution for \emph{(MOC)}. Then $W$ is generalized proximal solution
for \emph{(MOC)}.
\end{proposition}
\begin{proof}
We first check that (\ref{eq:proximal_equation}) is obtained. Let $(\xi^*,\mathbf{v^*}) \in D^*W_\uparrow(t,\mathbf{x},\mathbf{y})(\mathbf{w^*}).$ Then, $(\xi^*,\mathbf{v^*},-\mathbf{w^*})
\in  N_{\mathrm{gph}(W_\uparrow)}(t,\mathbf{x},\mathbf{y}).$ Hence, for all $(\xi,\mathbf{v},\mathbf{w}) \in T_{\mathrm{gph}(W_\uparrow)}(t,\mathbf{x},\mathbf{y})$,
\begin{equation} \label{eq:polarity}
  \xi \xi^* + \langle\mathbf{v} , \mathbf{v}^* \rangle + \langle \mathbf{w}, -\mathbf{w}^* \rangle \leq 0.
\end{equation}
From (\ref{eq:vfirst}), there exists $(\mathbf{\overline{f}},\mathbf{\overline{L}}) \in (\mathrm{FL})(\mathbf{x})$
such that
$$
 -\mathbf{\overline{L}} \in DW_\uparrow((t,\mathbf{x},\mathbf{y});(1,\mathbf{\overline{f}})),$$
 or
 $(1,\mathbf{\overline{f}},-\mathbf{\overline{L}}) \in T_{\mathrm{gph}(W_\uparrow)}(t,\mathbf{x},\mathbf{y})$. Hence,
 from (\ref{eq:polarity}), we get:
 $$\xi^* + \langle \mathbf{\overline{f}} , \mathbf{v}^* \rangle + \langle -\mathbf{\overline{L}}, -\mathbf{w}^* \rangle \leq 0,$$
 or
\begin{equation} \label{eq:inequality_1}
\xi^*+ \inf_{(\mathbf{f},\mathbf{L}) \in (\mathrm{FL})(\mathbf{x})} \langle
\mathbf{v^*},\mathbf{f}  \rangle + \langle \mathbf{w^*}, \mathbf{L} \rangle \leq 0.
\end{equation}
Also, from (\ref{eq:vsecond}), for all $(\mathbf{f},\mathbf{L}) \in (\mathrm{FL})(\mathbf{x})$,
$$
 \mathbf{L} \in DW_\uparrow((t,\mathbf{x},\mathbf{y});(-1,-\mathbf{f})),
$$
or
 $(-1,-\mathbf{f},\mathbf{L}) \in T_{\mathrm{gph}(W_\uparrow)}(t,\mathbf{x},\mathbf{y})$. Hence,
 from (\ref{eq:polarity}), we get:
 $$-\xi^* + \langle -\mathbf{f} , \mathbf{v}^* \rangle + \langle \mathbf{L}, -\mathbf{w}^* \rangle \leq 0,$$
 or
\begin{equation} \label{eq:inequality_2}
\xi^*+ \inf_{(\mathbf{f},\mathbf{L}) \in (\mathrm{FL})(\mathbf{x})} \langle
\mathbf{v^*},\mathbf{f}  \rangle + \langle \mathbf{w^*}, \mathbf{L} \rangle \geq 0.
\end{equation}
Combining (\ref{eq:inequality_1}) and (\ref{eq:inequality_2}) yields (\ref{eq:proximal_equation}). \\ \\
We turn out to checking that (\ref{eq:boundary_condition}) is obtained. We only check the first condition in (\ref{eq:boundary_condition}), as checking the
second condition is similar. Let $\mathbf{y} \in W_\uparrow(0,\mathbf{x})$. We have $\mathbf{y} =
\mathbf{\widetilde{y}} + \mathbf{d}$ with  $\mathbf{\widetilde{y}} \in W(0,\mathbf{x})$
and $\mathbf{d} \in P$. From (\ref{eq:vfirst}), there exists
$(\mathbf{\overline{f}},\mathbf{\overline{L}}) \in (\mathrm{FL})(\mathbf{x})$
such that
$$
 -\mathbf{\overline{L}} \in DW_\uparrow((0,\mathbf{x},\mathbf{\widetilde{y}});(1,\mathbf{\overline{f}})).
 $$
By definition of the contingent derivative, there exist sequences $t_k \rightarrow 1$, $\mathbf{f}_k \rightarrow \mathbf{\overline{f}},$
$\mathbf{L}_k \rightarrow \mathbf{\overline{L}},$ and $h_k \rightarrow 0^+$ such that
$$ \mathbf{\widetilde{y}}+ h_k (-\mathbf{L}_k) \in W_\uparrow(h_k t_k, \mathbf{x} + h_k \mathbf{f}_k).$$
Hence,
$$ \mathbf{y}+ h_k (-\mathbf{L}_k) \in W_\uparrow(h_k t_k, \mathbf{x} + h_k \mathbf{f}_k),$$
with $\mathbf{y}+ h_k (-\mathbf{L}_k) \rightarrow \mathbf{y}$, $h_k t_k \rightarrow 0^+$,
and $\mathbf{x} + h_k \mathbf{f}_k \rightarrow \mathbf{x}$. This shows that
$\mathbf{y} \in \limsup_{t \rightarrow  0^+, \ \mathbf{x'}\rightarrow \mathbf{x}} W_\uparrow(t,\mathbf{x'})$
and therefore the first condition in (\ref{eq:boundary_condition}) is obtained.
\end{proof}

\subsection{Uniqueness}

In this section, we first provide in Theorem \ref{th:prox1} a characterization of
set-valued maps satisfying (\ref{eq:proximal_equation})-(\ref{eq:tercon2}). The proof
of this theorem follows very closely the proof of Theorem 12.3.7  \cite[p.~460]{OptConTh5_book} for single objective optimal control problems. Using Theorem \ref{th:prox1} and the concept of extremal element maps, we are able to prove in
Corollary \ref{cor:conclusion} that
the set-valued return function $V$  is the unique generalized contingent and proximal solution for (MOC).\\
\begin{theorem} \label{th:prox1}
Let $W$ be a set-valued map from $I \times \mathbf{R}^n$ to $\mathbf{R}^p$. Assume the following
for $W$:\\
\begin{enumerate}
  \item $W_\uparrow$ is an outer semicontinuous set-valued map.
  \item For all $(t,\mathbf{x}) \in (0,T)\times \mathbf{R}^n$, $\mathbf{y} \in W(t,\mathbf{x})$, and
  $\mathbf{w^*} \in \mathbf{R}^p$ such that $D^*W_\uparrow(t,\mathbf{x},\mathbf{y})(\mathbf{w^*}) \neq \emptyset$, we have:
  \begin{equation} \label{eq:inwardpointing1}
   \forall (\xi^*,\mathbf{v^*}) \in D^*W_\uparrow(t,\mathbf{x},\mathbf{y})(\mathbf{w^*}), \
\xi^*+ \min_{(\mathbf{f},\mathbf{L}) \in (\mathrm{FL})(\mathbf{x})} \langle\mathbf{v^*},\mathbf{f}  \rangle + \langle \mathbf{w^*}, \mathbf{L} \rangle = 0.
\end{equation}
  \item For all $\mathbf{x} \in \mathbf{R}^n$, $$W_\uparrow(0,\mathbf{x}) \subset \limsup_{t \rightarrow 0^+, \ \mathbf{x'}\rightarrow \mathbf{x}} W_\uparrow(t,\mathbf{x'}) \ \mathrm{and} \
  W_\uparrow(T,\mathbf{x}) \subset \limsup_{t \rightarrow  T^-,\ \mathbf{x'}\rightarrow \mathbf{x}} W_\uparrow(t,\mathbf{x'}).$$
  \item For all $\mathbf{x} \in \mathbf{R}^n$, $W(T,\mathbf{x}) = \mathbf{0}.$\\
\end{enumerate}
Then, $\forall (t,\mathbf{x}) \in  I \times \mathbf{R}^n$,
\begin{equation} \label{eq:firstinclusion} V(t,\mathbf{x}) \subset W_\uparrow(t,\mathbf{x}), \end{equation}
and
\begin{equation} \label{eq:secondinclusion} W(t,\mathbf{x})\subset V_\uparrow(t,\mathbf{x}). \end{equation}
\end{theorem}

\begin{proof} \newline \newline
Proof of (\ref{eq:firstinclusion}): the proof of (\ref{eq:firstinclusion}) uses the Strong Invariance Theorem for Time-Varying
Systems (Theorem 12.2.4 \cite[pp.~449--452]{OptConTh5_book}). \\ \\
Let $(t,\mathbf{x}) \in  [0,T) \times \mathbf{R}^n$ and $\mathbf{u}(\cdot) \in \mathcal{U}$.
For $s \in  [0,T-t]$, define:\\
\begin{itemize}
  \item $\mathbf{\widetilde{x}}(s;t,\mathbf{x},\mathbf{u}(\cdot))
= \mathbf{x}(T-s;t,\mathbf{x},\mathbf{u}(\cdot)).$
  \item $\mathbf{\widetilde{a}}(s;t,\mathbf{x},\mathbf{u}(\cdot))
= \int_{T-s}^T \mathbf{L}(\mathbf{x}(\tau;t,\mathbf{x},\mathbf{u}(\cdot)),\mathbf{u}(\tau)) \ \mathrm{d}\tau.$
  \item  $(\widetilde{\mathrm{FL}})(\mathbf{x}) = \mathrm{cl}(\mathrm{co}( \{(-\mathbf{f}(\mathbf{x},\mathbf{u}),\mathbf{L}(\mathbf{x},\mathbf{u})), \ \mathbf{u} \in U\})).$
  \item $\widetilde{W}(s,\mathbf{x}) = W(T-s,\mathbf{x}).$\\
\end{itemize}
It is easy to check that $s \rightarrow (\mathbf{\widetilde{x}}
(s;t,\mathbf{x},\mathbf{u}(\cdot)),\mathbf{\widetilde{a}}(s;t,\mathbf{x},\mathbf{u}(\cdot)))$
is a solution of the following differential inclusion:
\begin{displaymath}
\left\{ \begin{array}{lcl}
    (\mathbf{\dot{x}}(s),\mathbf{\dot{a}}(s)) & \in & (\widetilde{\mathrm{FL}})(\mathbf{x}(s)), \  a.e. \  s \in [0,T-t],\\
    (\mathbf{x}(0),\mathbf{a}(0)) & = & (\mathbf{x}(T;t,\mathbf{x},\mathbf{u}(\cdot)),\mathbf{0}).
   \end{array} \right.
\end{displaymath}
Moreover, as  $(\xi^*,\mathbf{v^*}) \in D^*W_\uparrow(t,\mathbf{x},\mathbf{y})(\mathbf{w^*})$ implies that
$(-\xi^*,\mathbf{v^*}) \in D^*\widetilde{W}_\uparrow(t,\mathbf{x},\mathbf{y})(\mathbf{w^*})$, from
(\ref{eq:inwardpointing1}), we get:
\begin{equation} \label{eq:inwardpointing2}
\forall (\mathbf{f},\mathbf{L}) \in (\widetilde{\mathrm{FL}})(\mathbf{x}), \ \xi^*+  \langle\mathbf{v}^*,\mathbf{f}  \rangle - \langle \mathbf{w}^*, \mathbf{L} \rangle \leq 0.
\end{equation}
Define now the set-valued map $\widetilde{Q}$ from $[0,T-t]$ to $\mathbf{R}^{n+p}$ by
$$ \widetilde{Q}(s) = \{(\mathbf{x},\mathbf{a}), \ \mathbf{a} \in \widetilde{W}_\uparrow(s,\mathbf{x})\},$$
and consider the following constrained optimal control problem:
\begin{displaymath}
\left\{ \begin{array}{lcl}
        (\mathbf{\dot{x}}(s),\mathbf{\dot{a}}(s)) & \in &  (\widetilde{\mathrm{FL}})(\mathbf{x}(s)), \  a.e. \  s \in [0,T-t],\\
    (\mathbf{x}(0),\mathbf{a}(0)) & = & (\mathbf{x}(T;t,\mathbf{x},\mathbf{u}(\cdot)),\mathbf{0}),\\
        (\mathbf{x}(s),\mathbf{a}(s)) & \in &  \widetilde{Q}(s) \ \mathrm{for} \ \mathrm{all} \ s \in [0,T-t].
       \end{array} \right.
\end{displaymath}
Our objective is to apply the Strong Invariance Theorem for Time-Varying Systems to this problem. To do so,
we need to check that the conditions required to apply this theorem are satisfied. First,
as $W(T,\mathbf{x}) = \mathbf{0}$ by assumption 4, we have $\widetilde{Q}(0) = \mathbf{R}^n \times P$ and
hence $(\mathbf{x}(0),\mathbf{a}(0)) \in \widetilde{Q}(0)$. The condition
$(\mathbf{x}(0),\mathbf{a}(0))
\in \limsup_{s\rightarrow 0^+}  \widetilde{Q}(s)$ directly follows from assumption 3. Moreover, from the assumptions in \S \ref{s:P}, the
set-valued map $\widetilde{\mathrm{FL}}$ possesses the required properties, and $\mathrm{gph}(\widetilde{Q}) $ is closed
as $\mathrm{gph}(\widetilde{Q}) = \mathrm{gph}(\widetilde{W}_\uparrow)$ and $W_\uparrow$ is outer semicontinuous
by assumption 1. It remains to show the ``inward-pointing'' condition. Hence,
let $(s,\mathbf{x},\mathbf{a}) \in \mathrm{gph}(\widetilde{Q})$ with $s \in (0,T-t)$ at which
$N_{\mathrm{gph}(\widetilde{Q})}(s,\mathbf{x},\mathbf{a})$ is nonempty.
Take any vector $(\xi^*,\mathbf{v^*},\mathbf{w^*}) \in N_{\mathrm{gph}(\widetilde{Q})}(s,\mathbf{x},\mathbf{a}),$
we need to show that
$$ \xi^* + \max_{(\mathbf{f},\mathbf{L})\in (\widetilde{\mathrm{FL}})(\mathbf{x})} \langle (\mathbf{v^*},\mathbf{w^*}),(\mathbf{f},\mathbf{L}) \rangle \leq 0,$$
or
$$ \forall (\mathbf{f},\mathbf{L}) \in (\widetilde{\mathrm{FL}})(\mathbf{x}), \ \xi^* +  \langle \mathbf{v^*},\mathbf{f} \rangle + \langle \mathbf{w^*},\mathbf{L} \rangle \leq 0 .$$
As already mentioned,  $\mathrm{gph}(\widetilde{Q}) = \mathrm{gph}(\widetilde{W}_\uparrow).$ Hence, $N_{\mathrm{gph}(\widetilde{Q})}(s,\mathbf{x},\mathbf{a}) = N_{\mathrm{gph}(\widetilde{W}_\uparrow)}(s,\mathbf{x},\mathbf{a})$, and
by definition of the coderivative $(\xi^*,\mathbf{v^*}) \in D^*\widetilde{W}_\uparrow(s,\mathbf{x},\mathbf{a})(-\mathbf{w^*})$. Hence, from (\ref{eq:inwardpointing2}), we get:
$$ \forall (\mathbf{f},\mathbf{L}) \in (\widetilde{\mathrm{FL}})(\mathbf{x}), \ \xi^* +  \langle \mathbf{v^*},\mathbf{f} \rangle + \langle \mathbf{w^*},\mathbf{L} \rangle \leq 0,$$
which is what was needed to show.\\ \\
The Strong Invariance Theorem for Time-Varying Systems therefore gives that for all
$(t,\mathbf{x}) \in  [0,T) \times \mathbf{R}^n$, $\mathbf{u}(\cdot) \in \mathcal{U}$,
and $s \in[0,T-t]$,
$$ (\mathbf{\widetilde{x}}(s;t,\mathbf{x},\mathbf{u}(\cdot)),\mathbf{\widetilde{a}}(s;t,\mathbf{x},\mathbf{u}(\cdot))) \in \widetilde{Q}(s),$$
or $\mathbf{\widetilde{a}}(s;t,\mathbf{x},\mathbf{u}(\cdot)) \in \widetilde{W}_\uparrow(s,\mathbf{\widetilde{x}}(s;t,\mathbf{x},\mathbf{u}(\cdot)))$.
Hence, for $s = T-t$,
$$ \mathbf{\widetilde{a}}(T-t;t,\mathbf{x},\mathbf{u}(\cdot)) = \mathbf{J}(t,T,\mathbf{x},\mathbf{u}(\cdot)) \in \widetilde{W}_\uparrow(T-t,\mathbf{x}(T-t;t,\mathbf{x},\mathbf{u}(\cdot))) =  W_\uparrow(t,\mathbf{x}).$$
This holds for all control $\mathbf{u}(\cdot) \in \mathcal{U}$. Hence, the objective space $Y(t,\mathbf{x})$
satisfies the inclusion
$$ Y(t,\mathbf{x}) \subset W_\uparrow(t,\mathbf{x}),$$
or, as $W_\uparrow$ takes as values closed sets,
$$ V(t,\mathbf{x}) \subset \mathrm{cl}(Y(t,\mathbf{x})) \subset W_\uparrow(t,\mathbf{x}).$$
From assumption 4 and recalling that $ V(T,\mathbf{x}) = \{\mathbf{0}\}$, we finally get
(\ref{eq:firstinclusion}) for all $(t,\mathbf{x}) \in I \times \mathbf{R}^n$.\\ \\
Proof of (\ref{eq:secondinclusion}): the proof of (\ref{eq:secondinclusion}) uses the Weak Invariance Theorem
 for Time-Varying Systems (Theorem 12.2.2 \cite[pp.~446--448]{OptConTh5_book}). \\ \\
Take an arbitrary point $(t,\mathbf{x}) \in  [0,T) \times \mathbf{R}^n$ and
$\mathbf{y} \in W(t,\mathbf{x}).$ Define the set-valued map $Q$ from $[t,T]$
to $\mathbf{R}^{n+p}$ by
$$ Q(s) = \{(\mathbf{x},\mathbf{a}), \ \mathbf{a} \in W_\uparrow(s,\mathbf{x})\},$$
and consider the following optimal control problem:
\begin{equation} \label{eq:weakmoc}
\left\{ \begin{array}{lcl}
        (\mathbf{\dot{x}}(s),\mathbf{\dot{a}}(s)) & \in &  (\widetilde{\mathrm{FL}})(\mathbf{x}(s)), \  a.e. \  s \in [ t,T],\\
        (\mathbf{x}(t), \mathbf{a}(t)) & = & (\mathbf{x},\mathbf{y}),\\
        (\mathbf{x}(s),\mathbf{a}(s)) &  \in & Q(s)  \ \mathrm{for} \ \mathrm{all} \ s \in [t,T],
       \end{array} \right.
\end{equation}
where $(\widetilde{\mathrm{FL}})(\mathbf{x}) = \mathrm{cl}(\mathrm{co}( \{(\mathbf{f}(\mathbf{x},\mathbf{u}),-\mathbf{L}(\mathbf{x},\mathbf{u})), \ \mathbf{u} \in U\})).$
Our objective is to apply the Weak Invariance Theorem for Time-Varying Systems to this problem.
To do so,
we need to check that the conditions required to apply this theorem are satisfied. First, as $\mathbf{y} \in W(t,\mathbf{x})$, $(\mathbf{x}(t), \mathbf{a}(t))
\in Q(t)$. The condition
$(\mathbf{x}(0),\mathbf{a}(0))
\in \limsup_{s\rightarrow 0^+}  Q(s)$ directly follows from assumption 3. Moreover, from the assumptions in \S \ref{s:P}, the
set-valued map $\widetilde{\mathrm{FL}}$ possesses the required properties, and $\mathrm{gph}(Q) $ is closed
as $\mathrm{gph}(Q) = \mathrm{gph}(W_\uparrow)$ and $W_\uparrow$ is outer semicontinuous
by assumption 1. It remains to show the ''inward-pointing'' condition. Let $(s,\mathbf{x},\mathbf{a}) \in \mathrm{gph}(Q)$
with $s \in (t,T)$ at which $N_{\mathrm{gph}(Q)}(s,\mathbf{x},\mathbf{a})$ is nonempty.
Take any vector $(\xi^*,\mathbf{v^*},\mathbf{w^*}) \in N_{\mathrm{gph}(Q)}(s,\mathbf{x},\mathbf{a}),$ we need to show that
$$ \xi^* + \min_{(\mathbf{f},\mathbf{L})\in (\mathrm{\widetilde{FL}})(\mathbf{x})} \langle (\mathbf{v^*},\mathbf{w^*}),
(\mathbf{f},\mathbf{L}) \rangle \leq 0,$$
or
$$ \xi^* +  \min_{(\mathbf{f},\mathbf{L}) \in (\mathrm{\widetilde{FL}})(\mathbf{x})}
\langle \mathbf{v^*},\mathbf{f} \rangle + \langle \mathbf{w^*},\mathbf{L} \rangle \leq 0.$$
As already mentioned,  $\mathrm{gph}(Q) = \mathrm{gph}(W_\uparrow).$ Hence,
$N_{\mathrm{gph}(Q)}(s,\mathbf{x},\mathbf{a}) = N_{\mathrm{gph}(W_\uparrow)}(s,\mathbf{x},\mathbf{a})$, and
by definition of the coderivative $(\xi^*,\mathbf{v^*}) \in D^*W_\uparrow(s,\mathbf{x},\mathbf{a})(-\mathbf{w^*})$. Hence, from
 (\ref{eq:inwardpointing1}), we get:
  $$
   \xi^*+ \min_{(\mathbf{f},\mathbf{L}) \in (\mathrm{FL})(\mathbf{x})}  \langle \mathbf{v^*},\mathbf{f}  \rangle + \langle -\mathbf{w^*}, \mathbf{L} \rangle \leq 0,
$$
or
  $$
   \xi^*+ \min_{(\mathbf{f},\mathbf{L}) \in (\mathrm{\widetilde{FL}})(\mathbf{x})}  \langle \mathbf{v^*},\mathbf{f}  \rangle + \langle \mathbf{w^*}, \mathbf{L} \rangle \leq 0,
$$
which is what was needed to show.\\ \\
The Weak Invariance Theorem for Time-Varying Systems therefore gives that there exists a function $s \in [t,T] \rightarrow (\mathbf{x}(s),\mathbf{a}(s))$ solution to (\ref{eq:weakmoc}). Hence,
for all $s \in [t,T], \ (\mathbf{x}
(s),\mathbf{a}(s)) \in Q(s)$,
or $$\mathbf{a}(s) \in W_\uparrow(s,\mathbf{x}(s)).$$
Take $s = T$ to get:
$$\mathbf{a}(T) \in W_\uparrow(T,\mathbf{x}(T)) = P.$$
Now, we apply the Relaxation Theorem (Theorem 2.7.2 \cite[p.~96]{OptConTh5_book}) to $s \in [t,T] \rightarrow (\mathbf{x}(s),\mathbf{a}(s))$. Hence, there exists a sequence $\mathbf{u_n}(\cdot)$ such  that
$$ \lim_{n \rightarrow \infty} \mathbf{x}(\cdot;t,\mathbf{x},\mathbf{u_n}(\cdot)) = \mathbf{x}(\cdot) \ \mathrm{and} \ \lim_{n \rightarrow \infty} \mathbf{y} + \int_{t}^{\cdot} -\mathbf{L}(\mathbf{x}(s;t,
\mathbf{x},\mathbf{u_n}(\cdot)),\mathbf{u_n}(s)) \ \mathrm{d}s = \mathbf{a}(\cdot).$$
In particular, for $s = T$, we get:
$$ \lim_{n \rightarrow \infty}  \mathbf{y} - \int_{t}^{T} \mathbf{L}(\mathbf{x}(s;t,
\mathbf{x},\mathbf{u_n}(\cdot)),\mathbf{u_n}(s)) \ \mathrm{d}s = \mathbf{a}(T).$$
From above, we know that $\mathbf{a}(T) \in P$. Moreover, $\int_{t}^{T} \mathbf{L}(\mathbf{x}(s;t,
\mathbf{x},\mathbf{u_n}(\cdot)),\mathbf{u_n}(s)) \ \mathrm{d}s = \mathbf{J}(t,T,\mathbf{x},\mathbf{u_n}(\cdot))) \in Y(t,\mathbf{x})$.
Hence, $$\mathbf{y} \in \mathrm{cl}(Y(t,\mathbf{x})) + P = V(t,\mathbf{x}) + P = V_\uparrow(t,\mathbf{x}).$$
 This holds
for all $\mathbf{y} \in W(t,\mathbf{x}),$ hence $$ W(t,\mathbf{x})\subset V_\uparrow(t,\mathbf{x}).$$
From assumption 4 and recalling that $ V(T,\mathbf{x}) = \{\mathbf{0}\}$, we finally get
(\ref{eq:secondinclusion}) for all $(t,\mathbf{x}) \in I \times \mathbf{R}^n$.
\end{proof}\\

\begin{corollary} \label{cor:uniqueness}
Let $W$ be a set-valued map from $I \times \mathbf{R}^n$ to $\mathbf{R}^p$. Assume that $W$ is an extremal element map.
Then, under the assumptions of Theorem \emph{\ref{th:prox1}}, we have:
$$\forall (t,\mathbf{x}) \in  I \times \mathbf{R}^n, \ W(t,\mathbf{x}) =  V(t,\mathbf{x}).$$
\end{corollary}
\begin{proof}
Let $\mathbf{y_1} \in V(t,\mathbf{x}).$ From (\ref{eq:firstinclusion}),
$\mathbf{y_1} = \mathbf{w} + \mathbf{d_1}$ for some $\mathbf{w} \in W(t,\mathbf{x})$
and $\mathbf{d_1} \in P.$ From (\ref{eq:secondinclusion}), $\mathbf{w} = \mathbf{y_2} + \mathbf{d_2}$
for some $\mathbf{y_2} \in V(t,\mathbf{x})$
and $\mathbf{d_2} \in P.$ Hence, $\mathbf{y_1} = \mathbf{y_2} + \mathbf{d_1} + \mathbf{d_2}$. As $V$ is an
extremal element map, we must have $\mathbf{d_1} + \mathbf{d_2} = 0,$ which implies $\mathbf{d_1} = \mathbf{d_2} = 0$,
as $P$ is pointed. Hence, $\mathbf{y_1} \in  W(t,\mathbf{x})$ which shows that
$V(t,\mathbf{x}) \subset W(t,\mathbf{x})$. Inverting the role
of $V$ and $W$ yields $W(t,\mathbf{x}) \subset V(t,\mathbf{x})$,
and finally $W(t,\mathbf{x}) = V(t,\mathbf{x}).$
\end{proof} \\

Corollary \ref{cor:conclusion} summarizes the findings of this paper.\\

\begin{corollary} \label{cor:conclusion}
The set-valued return map $V$ is the unique  generalized contingent and proximal solution for \emph{(MOC)}
such that the set-valued map $V_\uparrow$ is outer semicontinuous.
\end{corollary}
\begin{proof}
In \S \ref{s:setvalued}, we have shown that $V$ is a generalized contingent solution for (MOC) and
that the set-valued map $V_\uparrow$ is outer semicontinuous. From Proposition \ref{prop:equivalence},
it follows that $V$ is also a generalized proximal solution for (MOC). Now, let $W$ be a generalized
proximal solution for (MOC) such that the set-valued map $W_\uparrow$ is outer semicontinuous.
By definition of the generalized proximal solution, W is an extremal element map and
satisfies the assumptions 2, 3, and 4 of Theorem  \ref{th:prox1}. Moreover,
$W_\uparrow$ is outer semicontinuous. Hence, Corollary \ref{cor:uniqueness} applies and we get
$W=V$.
\end{proof}

\section*{Acknowledgments}

The author wishes to thank Prof. Philip Loewen for helpful
discussions and encouragements.

\bibliography{E:/Papers/Bibliography/biblio}

\begin{thebibliography}{10}

\bibitem{SVA_book}
{\sc J.-P. Aubin and H.~Frankowska}, {\em Set-Valued Analysis}, Birkhauser,
  Boston, 1990.

\bibitem{OptConTh8_book}
{\sc M.~Bardi and I.~Capuzzo-Dolcetta}, {\em Optimal Control and Viscosity
  Solutions of Hamilton-Jacobi-Bellman Equations}, Birkhauser, Boston, 1997.

\bibitem{Bellaassali04_opt}
{\sc S.~Bellaassali and A.~Jourani}, {\em Necessary optimality conditions in
  multiobjective dynamic optimization}, SIAM J. Control Optim., 42 (2004),
  pp.~2043--2061.

\bibitem{Caroff08_MOC}
{\sc N.~Caroff}, {\em Multicriteria optimal control and vectorial
  hamilton-jacobi equation}, 6th International Conference Large-Scale
  Scientific Computing 2007, LNCS 4818, I. Lirkov, S. Margenov, and J.
  Wasniewski (Eds.), Springer, Berlin,  (2008), pp.~293--299.

\bibitem{Chen98_moo}
{\sc G.~Y. Chen and J.~Jahn}, {\em Optimality conditions for set-valued
  optimization problems}, Math. Methods Oper. Res., 48 (1998), pp.~187--200.

\bibitem{MOC_appl_6}
{\sc V.~Coverstone-Carroll, J.~W. Hartmann, and W.~J. Mason}, {\em Optimal
  multi-objective low-thrust spacecraft trajectories}, Comput. Methods Appl.
  Mech. Engrg., 186 (2000), pp.~387--402.

\bibitem{Guigue09}
{\sc A.~Guigue, M.~Ahmadi, M.~J.~D. Hayes, and R.~G. Langlois}, {\em A discrete
  dynamic programming approximation to the multiobjective deterministic finite
  horizon optimal control problem}, SIAM J. Control Optim., 48 (2009),
  pp.~2581--2599.

\bibitem{Guigue10}
{\sc A.~Guigue, M.~Ahmadi, R.~G. Langlois, and M.~J.~D. Hayes}, {\em Pareto
  optimality and multiobjective trajectory planning for a 7-dof redundant
  manipulator}, IEEE Transactions on Robotics, 26 (2010), pp.~1094 -- 1099.

\bibitem{MOC_appl_5}
{\sc B.-Z. Guo and B.~Sun}, {\em Numerical solution to the optimal feedback
  control of continuous casting process}, J. Glob. Optim., 39 (1998),
  pp.~171–--195.

\bibitem{SVA2_book}
{\sc J.~Jahn}, {\em Vector Optimization: Theory, Applications and Extensions},
  Springer-Verlag, Berlin, 2004.

\bibitem{Jahn97_moo}
{\sc J.~Jahn and R.~Rauh}, {\em Contingent epiderivatives and set-valued
  optimization}, Math. Methods Oper. Res., 46 (1997), pp.~193--211.

\bibitem{Kien08_opt}
{\sc B.~T. Kien, N.-C. Wong, and J.-C. Yao}, {\em Necessary conditions for
  multiobjective optimal control problems with free end-time}, SIAM J. Control
  Optim., 47 (2008), pp.~2251--2274.

\bibitem{NonlinearMOO3_book}
{\sc D.~T. Luc}, {\em Theory of Vector Optimization}, Springer-Verlag, New
  York, 1989.

\bibitem{Mordukhovich93_moo}
{\sc B.~Mordukhovich}, {\em Complete characterization of openness, metric
  regularity, and lipschitzian properties of multifunctions}, Transactions of
  the American Mathematical Society, 340 (1993), pp.~1--35.

\bibitem{SVA3_book}
{\sc R.~T. Rockafellar and R.~J.-B. Wets}, {\em Variational Analysis},
  Springer-Verlag, Berlin, 1998.

\bibitem{NonlinearMOO1_book}
{\sc Y.~Sawaragi, H.~Nakayama, and T.~Tanino}, {\em Theory of Multiobjective
  Optimization}, Academic Press, Inc., Orlando, FL, 1985.

\bibitem{Tanino88_moo}
{\sc T.~Tanino}, {\em Sensitivity analysis in multiobjective optimization}, J.
  Optim. Theory Appl., 56 (1988), pp.~479--499.

\bibitem{OptConTh5_book}
{\sc R.~Vinter}, {\em Optimal Control}, Birkauser, Boston, 2000.

\bibitem{Yu74_moo}
{\sc P.~L. Yu}, {\em Cone convexity, cone extreme points, and nondominated
  solutions in decision problems with multiobjectives}, J. Optim. Theory Appl.,
  14 (1974), pp.~319--377.

\bibitem{Zhu01_opt}
{\sc Q.~J. Zhu}, {\em Hamiltonian necessary conditions for a multiobjective
  optimal control problem with endpoint constraints}, SIAM J. Control Optim.,
  39 (2001), pp.~97--112.

\end{thebibliography}
\bibliographystyle{siam}

\end{document}